\documentclass[a4paper,fleqn,11pt]{article}

\usepackage[a4paper,top=3cm, bottom=3cm, left=3cm, right=3cm]{geometry}
\usepackage[shortlabels]{enumitem}
\usepackage[pdftex,ocgcolorlinks]{hyperref}
\usepackage[affil-it]{authblk}

\usepackage{amsmath}
\usepackage{amsthm}
\usepackage{amssymb}
\usepackage{accents}
\usepackage{cleveref}

\theoremstyle{plain}
\newtheorem{theorem}{Theorem}[section]
\newtheorem{lemma}[theorem]{Lemma}
\newtheorem{proposition}[theorem]{Proposition}
\newtheorem{corollary}[theorem]{Corollary}

\newtheorem{definition}[theorem]{Definition}

\crefname{enumi}{property}{properties}

\newcommand{\entropy}{H}
\newcommand{\mutualinformation}{I}
\newcommand{\distributions}[1][]{\mathcal{P}_{#1}}
\newcommand{\typeclass}[2]{T^{#1}_{#2}}
\newcommand{\norm}[2][]{\left\|#2\right\|_{#1}}
\newcommand{\ball}[2]{B_{#1}(#2)}
\DeclareMathOperator{\Aut}{Aut}
\DeclareMathOperator{\support}{supp}
\DeclareMathOperator{\probability}{Pr}
\newcommand{\setbuild}[2]{\left\{#1\middle|#2\right\}}

\newcommand{\strongproduct}{\boxtimes}
\newcommand{\bigstrongproduct}{\sideset{}{^\boxtimes}\prod}
\newcommand{\costrongproduct}{*}
\newcommand{\lexproduct}{\circ}
\newcommand{\graphs}{\mathcal{G}}
\newcommand{\probgraphs}{\mathcal{G}_{\textnormal{prob}}}
\newcommand{\typegraph}[3]{#1^{\strongproduct #2}[\typeclass{#2}{#3}]} % G,n,P
\DeclareMathOperator{\ev}{ev}
\renewcommand{\complement}[1]{\overline{#1}}
\newcommand{\thetabody}{\mathrm{TH}}
\newcommand{\vertexpackingpolytope}{\mathrm{VP}}
\newcommand{\fractionalvertexpackingpolytope}{\mathrm{FVP}}

\title{Probabilistic refinement of the asymptotic spectrum of graphs}
\author[1,2]{P\'eter Vrana}
\affil[1]{Institute of Mathematics, Budapest University of Technology and Economics, Egry~J\'ozsef u.~1., Budapest, 1111 Hungary.}
\affil[2]{MTA-BME Lend\"ulet Quantum Information Theory Research Group}

\begin{document}
\maketitle
\begin{abstract}
The asymptotic spectrum of graphs, introduced by Zuiddam (arXiv:1807.00169, 2018), is the space of graph parameters that are additive under disjoint union, multiplicative under the strong product, normalized and monotone under homomorphisms between the complements. He used it to obtain a dual characterization of the Shannon capacity of graphs as the minimum of the evaluation function over the asymptotic spectrum and noted that several known upper bounds, including the Lov\'asz number and the fractional Haemers bounds are in fact elements of the asymptotic spectrum (spectral points).

We show that every spectral point admits a probabilistic refinement and characterize the functions arising in this way. This reveals that the asymptotic spectrum can be parameterized with a convex set and the evaluation function at every graph is logarithmically convex. One consequence is that for any incomparable pair of spectral points $f$ and $g$ there exists a third one $h$ and a graph $G$ such that $h(G)<\min\{f(G),g(G)\}$, thus $h$ gives a better upper bound on the Shannon capacity of $G$. In addition, we show that the (logarithmic) probabilistic refinement of a spectral point on a fixed graph is the entropy function associated with a convex corner.
\end{abstract}

\section{Introduction}

The Shannon capacity of a graph $G$ is \cite{shannon1956zero}
\begin{equation}
\Theta(G)=\lim_{n\to\infty}\sqrt[n]{\alpha(G^{\strongproduct n})},
\end{equation}
where $\alpha$ denotes the independence number and $G^{\strongproduct n}$ is the $n$th strong power (see \Cref{sec:preliminaries} for definitions). In the context of information theory, the optimal rate of zero-error communication over a noisy classical channel is equal to the Shannon capacity of its confusability graph.

In \cite{zuiddam2018asymptotic} Zuiddam introduced the asymptotic spectrum of graphs as follows. Let $\graphs$ denote the set of isomorphism classes of finite undirected simple graphs. The asymptotic spectrum of graphs $\Delta(\graphs)$ is the set of functions $f:\graphs\to\mathbb{R}_{\ge 0}$ which satisfy for all $G,H\in\graphs$
\begin{enumerate}[({S}1)]
\item\label{it:spectraladditive} $f(G\sqcup H)=f(G)+f(H)$ (additive under disjoint union)
\item\label{it:spectralmultiplicative} $f(G\strongproduct H)=f(G)f(H)$ (multiplicative under the strong product)
\item\label{it:spectralmonotone} if there is a homomorphism $\complement{H}\to\complement{G}$ between the complements then $f(H)\le f(G)$
\item\label{it:spectralnormalised} $f(K_1)=1$.
\end{enumerate}
Elements of $\Delta(\graphs)$ are also called spectral points. Using the theory of asymptotic spectra, developed by Strassen in \cite{strassen1988asymptotic}, he found the following characterization of the Shannon capacity:
\begin{equation}\label{eq:minDelta}
\Theta(G)=\min_{f\in\Delta(\graphs)}f(G).
\end{equation}
A number of well-studied graph parameters turn out to be spectral points: the Lov\'asz theta number $\vartheta(G)$ \cite{lovasz1979shannon}, the fractional clique cover number $\complement{\chi}_f(G)$, the complement of the projective rank $\complement{\xi}_f(G)=\xi_f(\complement{G})$ \cite{cubitt2014bounds}, and the fractional Haemers bound $\mathcal{H}^\mathbb{F}_f(G)$ over any field $\mathbb{F}$ \cite{haemers1978upper,blasiak2013graph,bukh2018fractional}. The latter gives rise to an infinite family of distinct points. $\complement{\chi}_f(G)$ is also the maximum of the spectral points. In fact, both this and \cref{eq:minDelta} remains true if we allow optimization over the larger set of functions subject only to \cref{it:spectralmultiplicative,it:spectralmonotone,it:spectralnormalised} \cite[8.1. Example]{fritz2017resource}.

In \cite{csiszar1981capacity} Csisz\'ar and K\"orner introduced a refinement of the Shannon capacity, imposing that the independent set consists of sequences with the same frequency for each vertex of $G$, in the limit approaching a prescribed probability distribution $P$ on the vertex set $V(G)$. Their definition is equivalent to
\begin{equation}\label{eq:constantcompositioncapacity}
\Theta(G,P)=\lim_{\epsilon\to 0}\limsup_{n\to\infty}\sqrt[n]{\alpha(\typegraph{G}{n}{\ball{\epsilon}{P}})},
\end{equation}
where $\typeclass{n}{\ball{\epsilon}{P}}\subseteq V(G)^n$ is the set of those sequences whose type (empirical distribution) is $\epsilon$-close to $P$ and $\typegraph{G}{n}{\ball{\epsilon}{P}}$ is the subgraph induced by this subset. Some properties are more conveniently expressed in terms of $C(G,P)=\log\Theta(G,P)$, which is also called the Shannon capacity.

In information theory, the independent sets in a type class are constant composition codes for zero-error communication, while similar notions in graph theory are sometimes called probabilistic refinements or ``within a type'' versions. To avoid proliferation of notations, we adopt the convention that graph parameters and their probabilistic refinements (defined using strong products) are denoted with the same symbol, even if alternative notation is in use elsewhere.

The aim of this paper is to gain a better understanding of $\Delta(\graphs)$ by studying the probabilistic refinements of spectral points, focusing on those properties which follow from \cref{it:spectraladditive,it:spectralmultiplicative,it:spectralmonotone,it:spectralnormalised} and thus are shared by all of them. Some of these properties were already known to be true for specific ones. 

\subsection{Results}

Before stating the main results we introduce some terminology. A probabilistic graph $(G,P)$ is a nonempty graph $G$ together with a probability measure $P$ on $V(G)$ (notation: $P\in\distributions(V(G))$). Two probabilistic graphs are isomorphic if there is an isomorphism between the underlying graphs that is measure preserving. Let $\probgraphs$ denote the set of isomorphism classes of probabilistic graphs.
\begin{theorem}\label{thm:spectralpointtoprobabilistic}
Let $f\in\Delta(\graphs)$. Then for any probabilistic graph $(G,P)$ the
limit
\begin{equation}
f(G,P):=\lim_{\epsilon\to 0}\limsup_{n\to\infty}\sqrt[n]{f(\typegraph{G}{n}{\ball{\epsilon}{P}})}
\end{equation}
exists. Consider $F(G,P)=\log f(G,P)$ as a function $F:\probgraphs\to\mathbb{R}$. It satisfies the following properties: 
\begin{enumerate}[({P}1)]
\item\label{it:probconcave} for any graph $G$ the map $P\mapsto F(G,P)$ is concave
\item\label{it:probmultiplicative} if $G,H$ are graphs and $P\in\distributions(V(G)\times V(H))$ then
\begin{equation}
F(G\strongproduct H,P)\le F(G,P_G)+F(H,P_H)\le F(G\strongproduct H,P)+\mutualinformation(G:H)_P
\end{equation}
where $P_G$, $P_H$ denote the marginals of $P$ on $V(G)$ and $V(H)$, $\mutualinformation(G:H)_P=\entropy(P_G)+\entropy(P_H)-\entropy(P)$ is the mutual information with $\entropy(P)$ the Shannon entropy
\item\label{it:probadditive} if $G,H$ are graphs, $P_G\in\distributions(V(G))$, $P_H\in\distributions(V(H))$ and $p\in[0,1]$ then
\begin{equation}
F(G\sqcup H,pP_G\oplus(1-p)P_H)=pF(G,P_G)+(1-p)F(H,P_H)+h(p)
\end{equation}
where $h(p)=-p\log p-(1-p)\log(1-p)$
\item\label{it:probmonotone} if $\varphi:\complement{H}\to\complement{G}$ is a homomorphism and $P\in\distributions(V(H))$ then $F(H,P)\le F(G,\varphi_*(P))$
\end{enumerate}
and $f$ can be recovered as $f(G)=\max_{P\in\distributions(V(G))}f(G,P)$.
\end{theorem}

Unsurprisingly, it turns out that the following counterpart of \cref{eq:minDelta} for probabilistic graphs is true:
\begin{equation}\label{eq:minDeltaprobabilistic}
\Theta(G,P)=\min_{f\in\Delta(\graphs)}f(G,P).
\end{equation}

We prove the following converse to \Cref{thm:spectralpointtoprobabilistic}.
\begin{theorem}\label{thm:probabilistictospectralpoint}
Let $F:\probgraphs\to\mathbb{R}$ be a map satisfying \cref{it:probconcave,it:probmultiplicative,it:probadditive,it:probmonotone}. Consider the function $f:\graphs\to\mathbb{R}_{\ge0}$
\begin{equation}
f(G)=\begin{cases}
\max_{P\in\distributions(V(G))}2^{F(G,P)} & \text{if $G$ is nonempty}  \\
0 & \text{if $G$ is empty.}
\end{cases}
\end{equation}
Then $f\in\Delta(\graphs)$ and its logarithmic probabilistic refinement is $F$.
\end{theorem}

\Cref{thm:spectralpointtoprobabilistic,thm:probabilistictospectralpoint} set up a bijection between $\Delta(\graphs)$ and the set of functions satisfying \cref{it:probconcave,it:probmultiplicative,it:probadditive,it:probmonotone}. The inequalities defining the latter are affine, therefore it is a convex subset of the space of all functions on $\probgraphs$. Translating back to functions on $\graphs$, it follows that e.g. the graph parameter
\begin{equation}\label{eq:LovaszHaemers}
\begin{split}
f_{1/2}(G)
 & = \max_{P\in\distributions(V(G))}\sqrt{\vartheta(G,P)\mathcal{H}^{\mathbb{F}_2}_f(G,P)}  \\
 & = \max_{P\in\distributions(V(G))}\lim_{\epsilon\to0}\lim_{n\to\infty}\sqrt[2n]{\vartheta(\typegraph{G}{n}{\ball{\epsilon}{P}})\mathcal{H}^{\mathbb{F}_2}_f(\typegraph{G}{n}{\ball{\epsilon}{P}})}
\end{split}
\end{equation}
is an element of $\Delta(\graphs)$. Moreover, the function $f\mapsto \log f(G)$ is the maximum of affine functions for any fixed graph $G$, therefore it is convex. This allows us to find examples of graphs where a combined function like in \cref{eq:LovaszHaemers} gives a strictly better bound than the two spectral points involved.

In addition, we prove analogues of some of the properties that were previously known for specific spectral points. These include subadditivity with respect to the intersection; the value on the join of two graphs; and a characterization of multiplicativity under the lexicographic product. We introduce for each spectral point a complementary function and find a characterization of the Witsenhausen rate \cite{witsenhausen1976zero} and of the complementary graph entropy \cite{korner1973two}.

The probabilistic refinement of the fractional clique cover number (also known as the graph entropy of the complement) is the entropy with respect to the vertex packing polytope \cite{csiszar1990entropy}. Similarly, the probabilistic refinement of the Lov\'asz number is also the entropy with respect to a convex corner \cite{marton1993shannon}, called the theta body \cite{grotschel1986relaxations}. We show that this property is shared by every spectral point, and give another characterization of the probabilistic refinements as the entropy functions associated with certain convex corner-valued functions on $\graphs$.

\subsection{Organization of this paper}

In \Cref{sec:preliminaries} we collect basic definitions and facts from graph theory and information theory, in particular those which are central to the method of types. \Cref{sec:probabilisticrefinement} contains the proof of \Cref{thm:spectralpointtoprobabilistic,thm:probabilistictospectralpoint}. In \Cref{sec:furtherproperties} we discuss a number of properties that have been known for specific spectral points and are true for all (or at least a large subset) of them. These include subadditivity under intersection of graphs with common vertex set and the behaviour under graph join and lexicographic product. We also put some notions related to graph entropy and complementary graph entropy into our more general context. In \Cref{sec:convexcorners} we connect our results to the theory of convex corners. 

\section{Preliminaries}\label{sec:preliminaries}

Every graph in this paper is assumed to be a finite simple undirected graph. The vertex set of a graph $G$ is $V(G)$ and its edge set is $E(G)\subseteq\binom{V(G)}{2}$. The complement $\complement{G}$ is the graph with the same vertex set and edge set $E(G)=\binom{V(G)}{2}\setminus E(G)$. Given a graph $G$ and a subset $S\subseteq V(G)$ the induced subgraph $G[S]$ is the graph with vertex set $V(G[S])=S$ and edge set $E(G[S])=E(G)\cap\binom{S}{2}$. We write $G\subseteq G'$ when $V(G)=V(G')$ and $E(G)\subseteq E(G')$. This is a partial order and the complement operation is order reversing. The complete graph on a set $S$ is $K_S$, for $S=[n]$ the notation is simplified to $K_n$.

For graphs $G$ and $H$ the disjoint union $G\sqcup H$ has vertex set $V(G)\sqcup V(H)$ and edge set $E(G)\sqcup E(H)$. The strong product $G\strongproduct H$ has vertex set $V(G)\times V(H)$ and $\{(g,h),(g',h')\}\in E(G\strongproduct H)$ iff ($\{g,g'\}\in E(G)$ or $g=g'$) and ($\{h,h'\}\in E(H)$ or $h=h'$) but $(g,h)\neq(g',h')$. The join and the costrong product are
\begin{align}
G+H & = \complement{\complement{G}\sqcup\complement{H}}  \\
G\costrongproduct H & = \complement{\complement{G}\strongproduct\complement{H}}.
\end{align}
We use the notation $G^{\strongproduct n}=G\strongproduct G\strongproduct\cdots\strongproduct G$ ($n$ operands), and similarly for other associative binary operations. The lexicographic product $G\lexproduct H$ has vertex set $V(G)\times V(H)$ and $\{(g,h),(g',h')\}\in E(G\strongproduct H)$ iff $\{g,g'\}\in E(G)$ or ($g=g'$ and $\{h,h'\}\in E(H)$). The lexicographic product satisfies $\complement{G\lexproduct H}=\complement{G}\lexproduct\complement{H}$ and the three types of products are ordered as $G\strongproduct H\subseteq G\lexproduct H\subseteq G\costrongproduct H$.

A graph homomorphism $\varphi:G\to H$ is a function $\varphi:V(G)\to V(H)$ such that $\{g,g'\}\in E(G)\implies\{\varphi(g),\varphi(g')\}\in E(H)$ for all $g,g'\in V(G)$. An isomorphism is a homomorphism which is a bijection between the vertex sets and its inverse is also a homomorphism. $G$ and $H$ are isomorphic if there is an isomorphism between them. The set of isomorphism classes of graphs is denoted by $\graphs$. The set of isomorphisms $G\to G$ is $\Aut(G)$. We write $H\le G$ if there is a homomorphism $\varphi:\complement{H}\to\complement{G}$. In particular, $G[S]\le G$ for any $S\subseteq V(G)$, because the inclusion of an induced subgraph is a homomorphism and passing to induced subgraphs commutes with complementation.

A probability distribution $P$ on a finite set $\mathcal{X}$ is a function $P:\mathcal{X}\to[0,1]$ satisfying $\sum_{x\in\mathcal{X}}P(x)=1$. The support of $P$ is $\support P=P^{-1}((0,1])$. For $n\in\mathbb{N}_{>0}$, $P$ is said to be an $n$-type if $nP(x)\in\mathbb{N}$ for all $x\in\mathcal{X}$. The set of probability distributions on $\mathcal{X}$ will be denoted by $\distributions(\mathcal{X})$, the set of $n$-types by $\distributions[n](\mathcal{X})$ and
\begin{equation}
\distributions[\mathbb{Q}](\mathcal{X})=\bigcup_{n=1}^\infty\distributions[n](\mathcal{X}).
\end{equation}
The latter is a dense subset of $\distributions(\mathcal{X})$, equipped with the subspace topology from the Euclidean space $\mathbb{R}^\mathcal{X}$. $\ball{\epsilon}{P}$ denotes the open $\epsilon$-ball in $\distributions(\mathcal{X})$ centered at $P$ with respect to the total variation distance.

For an $n$-type $P\in\distributions[n](\mathcal{X})$ the type class $\typeclass{n}{P}\subseteq\mathcal{X}^n$ is the set of strings in which $x$ occurs exactly $nP(x)$ times for all $x\in\mathcal{X}$. More generally, for a subset $\mathcal{U}\subseteq\distributions(\mathcal{X})$ we define
\begin{equation}
\typeclass{n}{\mathcal{U}}=\bigcup_{P\in\mathcal{U}\cap\distributions[n](\mathcal{X})}\typeclass{n}{P}.
\end{equation}
The number of type classes satisfies $|\distributions[n](\mathcal{X})|\le(n+1)^{|\mathcal{X}|}$ \cite[Lemma 2.2]{csiszar2011information}.

The (Shannon) entropy of a probability distribution $P\in\distributions(\mathcal{X})$ is
\begin{equation}
\entropy(P)=-\sum_{x\in\mathcal{X}}P(x)\log P(x),
\end{equation}
where $\log$ is to base $2$ and by convention $0\log0=0$ (justified by continuous extension). A special case is the entropy of a Bernoulli distribution, $h(p)=-p\log p-(1-p)\log(1-p)$. When $P\in\distributions[n](\mathcal{X})$ we have the cardinality estimates \cite[Lemma 2.3]{csiszar2011information}
\begin{equation}\label{eq:typeclasssize}
\frac{1}{(n+1)^{|\mathcal{X}|}}2^{n\entropy(P)}\le |\typeclass{n}{P}|\le 2^{n\entropy(P)}.
\end{equation}
The relative entropy between two Bernoulli distributions is $d(p\|q)=p\log\frac{p}{q}+(1-p)\log\frac{1-p}{1-q}$, which satisfies $d(p\|q)\ge 0$.

When $\mathcal{X}$ and $\mathcal{Y}$ are finite sets and $P\equiv P_{XY}\in\distributions(\mathcal{X}\times\mathcal{Y})$, the distributions $P_X\in\distributions(\mathcal{X})$ and $P_Y\in\distributions(\mathcal{Y})$ given by
\begin{align}
P_X(x) & = \sum_{y\in \mathcal{Y}}P(x,y)  &
P_Y(x) & = \sum_{x\in \mathcal{X}}P(x,y)
\end{align}
are called the marginals of $P$. The mutual information is $\mutualinformation(X:Y)_P=\entropy(P_X)+\entropy(P_Y)-\entropy(P)$. $P_X\otimes P_Y$ denotes the probability distribution on $\mathcal{X}\times\mathcal{Y}$ given by $(P_X\otimes P_Y)(x,y)=P_X(x)P_Y(y)$, while for $p\in[0,1]$, $pP_X\oplus (1-p)P_Y$ denotes the distribution on $\mathcal{X}\sqcup\mathcal{Y}$ defined as
\begin{equation}
(pP_X\oplus (1-p)P_Y)(x)=\begin{cases}
pP_X(x) & \text{if $x\in\mathcal{X}$}  \\
(1-p)P_Y(y) & \text{if $y\in\mathcal{Y}$.}
\end{cases}
\end{equation}

If $f:\mathcal{X}\to\mathcal{Y}$ is a function between finite sets and $P\in\distributions(\mathcal{X})$, then the pushforward is the distribution $f_*(P)\in\distributions(\mathcal{Y})$ defined as
\begin{equation}
f_*(P)(y)=\sum_{x\in f^{-1}(y)}P(x).
\end{equation}

The probabilistic refinement of a graph parameter $Q:\graphs\to\mathbb{R}_{\ge0}$ is $\displaystyle Q(G,P)=\lim_{\epsilon\to 0}\lim_{n\to\infty}\sqrt[n]{Q(\typegraph{G}{n}{\ball{\epsilon}{P}})}$ whenever the limit exists, where $G$ is a nonempty graph $G$ and $P\in\distributions(V(G))$. In particular, existence is guaranteed if $Q$ is $\strongproduct$-supermultiplicative and nonincreasing under taking induced subgraphs. In all the examples in this paper, when $P\in\distributions[n](V(G))$, this quantity is the same as $\displaystyle\lim_{k\to\infty}\sqrt[kn]{Q(\typegraph{G}{kn}{P})}$.

\section{Probabilistic refinement of spectral points}\label{sec:probabilisticrefinement}

In this section we prove \Cref{thm:spectralpointtoprobabilistic,thm:probabilistictospectralpoint}. We let $f$ be an arbitrary fixed element of $\Delta(\graphs)$, the same symbol is used for its probabilistic refinement and $F(G,P)=\log f(G,P)$.
\begin{lemma}\label{lem:transitiveinduced}
Let $H$ be a vertex-transitive graph and $S\subseteq V(H)$. Then $H\le\complement{K_N}\strongproduct H[S]$ with
\begin{equation}
N=\left\lfloor\frac{|V(H)|}{|S|}\ln|V(H)|\right\rfloor+1.
\end{equation}
\end{lemma}
\begin{proof}
The proof is essentially a folklore argument. Draw $\pi_1,\ldots,\pi_N$ at random independently and uniformly from $\Aut(H)$. Define $m:[N]\times S\to V(H)$ as $m(i,u)=\pi_i^{-1}(u)$. For any $v\in V(H)$ and $i\in[N]$,
\begin{equation}
\probability[v\in m(\{i\}\times S)]=\probability[v\in \pi_i^{-1}(S)]=\probability[\pi_i(v)\in S]=\frac{|S|}{|V(H)|},
\end{equation}
because $\pi_i(v)$ is uniformly distributed on $V(H)$ by vertex-transitivity. For fixed $v$ and varying $i$ these events are independent, therefore
\begin{equation}
\begin{split}
\probability[\exists v\in V(H):v\notin m([N]\times S)]
 & \le |V(H)|
\left(1-\frac{|S|}{|V(H)|}\right)^N  \\
 & \le e^{-N\frac{|S|}{|V(H)|}+\ln|V(H)|}<1.
\end{split}
\end{equation}
Thus $m$ is surjective for some choice of the permutations. Fix such a choice and let $\varphi:V(H)\to[N]\times S$ be an arbitrary right inverse of $m$. Suppose that $u,v\in V(H)$ such that $\{u,v\}\notin E(H)$ and let $\varphi(u)=(i,u')$, $\varphi(v)=(j,v')$. If $i\neq j$ then $\{i,j\}$ is not an edge in $\complement{K_N}$, therefore $\{\varphi(u),\varphi(v)\}\notin E(\complement{K_N}\strongproduct H[S])$. Otherwise $\{u',v'\}=\{\pi_i(u),\pi_i(v)\}\notin E(H)$ since $\pi_i$ is an automorphism, therefore $\{\varphi(u),\varphi(v)\}\notin E(\complement{K_N}\strongproduct H[S])$. This proves that $\varphi:\complement{H}\to\complement{\complement{K_N}\strongproduct H[S]}$ is a homomorphism.
\end{proof}

\begin{lemma}\label{prop:convexcombinationbounds}
Let $G$ be a graph, $m,n\in\mathbb{N}$, $P\in\distributions[m](V(G))$ and $Q\in\distributions[n](V(G))$. Then
\begin{equation}
\typegraph{G}{m}{P}\strongproduct\typegraph{G}{n}{Q}\le\typegraph{G}{(m+n)}{\frac{mP+nQ}{m+n}}\le \complement{K_N}\strongproduct \typegraph{G}{m}{P}\strongproduct\typegraph{G}{n}{Q}
\end{equation}
for some $N\in\mathbb{N}$ satisfying
\begin{equation}
N\le(n+1)^{2|V(G)|}(m+1)^{2|V(G)|}2^{(n+m)\entropy(\frac{mP+nQ}{m+n})-m\entropy(P)-n\entropy(Q)}
\end{equation}
\end{lemma}
\begin{proof}
We start with the first inequality. Both sides can be represented as induced subgraphs of $G^{\strongproduct(m+n)}$, on the vertex sets $\typeclass{m}{P}\times\typeclass{n}{Q}$ and $\typeclass{m+n}{\frac{mP+nQ}{m+n}}$, respectively. Since $\typeclass{m}{P}\times\typeclass{n}{Q}\subseteq\typeclass{m+n}{\frac{mP+nQ}{m+n}}$, the left hand side is an induded subgraph of the right hand side.

For the second inequality we apply \Cref{lem:transitiveinduced} to the graph $\typegraph{G}{(m+n)}{\frac{mP+nQ}{m+n}}$ and the subset $T^m_P\times T^n_Q$ of its vertex set. The upper bound on the resulting $N$ follows from the (crude) estimate
\begin{equation}
\begin{split}
N
 & = \left\lfloor\frac{\Big|\typeclass{m+n}{\frac{mP+nQ}{m+n}}\Big|}{|\typeclass{m}{P}\times\typeclass{n}{Q}|}\ln\Big|\typeclass{m+n}{\frac{mP+nQ}{m+n}}\Big|\right\rfloor+1  \\
 & \le \left\lfloor(m+1)^{|V(G)|}(n+1)^{|V(G)|}2^{(m+n)\entropy(\frac{mP+nQ}{m+n})-m\entropy(P)-n\entropy(Q)}(m+n)\ln|V(G)|\right\rfloor+1  \\
 & \le (m+1)^{2|V(G)|}(n+1)^{2|V(G)|}2^{(m+n)\entropy(\frac{mP+nQ}{m+n})-m\entropy(P)-n\entropy(Q)}.
\end{split}
\end{equation}
\end{proof}

\begin{proposition}\label{prop:probabilisticrefinement}
For every nonempty graph $G$ and $P\in\distributions[m](V(G))$ we have
\begin{equation}
\lim_{k\to\infty}\frac{1}{km}\log f(\typegraph{G}{km}{P})=\sup_{k\in\mathbb{N}}\frac{1}{km}\log f(\typegraph{G}{km}{P}).
\end{equation}
This expression defines a uniformly continuous function $F(G,\cdot)$ on $\distributions[\mathbb{Q}](V(G))$, therefore has a unique continuous extension to $\distributions(V(G))$, which we denote by the same symbol. Moreover,
\begin{enumerate}[(i)]
\item\label{it:Fbounded} $0\le F(G,P)\le\entropy(P)$
\item\label{it:Fconcave} $P\mapsto F(G,P)$ is concave (\cref{it:probconcave})
\item\label{it:HFconcave} $P\mapsto \entropy(P)-F(G,P)$ is concave
\item\label{it:Fcontinuous} $F$ satisfies the continuity estimate
\begin{equation}\label{eq:continuityestimate}
\begin{split}
|F(G,P)-F(G,Q)|
 & \le\frac{\norm[1]{P-Q}}{2}\log(|V(G)|-1)+h\left(\frac{\norm[1]{P-Q}}{2}\right)  \\
 & \quad+2\left(1-h\left(\frac{2+\norm[1]{P-Q}}{4}\right)\right).
\end{split}
\end{equation}
\end{enumerate}
\end{proposition}
\begin{proof}
It is enough to establish existence of the limit and verify \cref{it:Fbounded,it:Fconcave,it:HFconcave,it:Fcontinuous} on $\distributions[\mathbb{Q}](V(G))$. \Cref{it:Fcontinuous} will then imply uniform continuity, hence existence of the continuous extension, which is unique since $\distributions[\mathbb{Q}](V(G))$ is dense in $\distributions(V(G))$.

Let $m\in\mathbb{N}$ and $P\in\distributions[m](V(G))$. Let $a_k=\log f(\typegraph{G}{km}{P})$. By the first inequality of \Cref{prop:convexcombinationbounds}, $\typegraph{G}{k_1m}{P}\strongproduct \typegraph{G}{k_2m}{P}\le \typegraph{G}{(k_1+k_2)m}{P}$. Apply $f$ to both sides. Using that $f$ is monotone, multiplicative under the strong product, and $f(\complement{K_N})=N$ we get
\begin{equation}
0\le a_{k_1}+a_{k_2}\le a_{k_1+k_2}\le\log|\typeclass{(k_1+k_2)m}{P}|\le(k_1+k_2)m\entropy(P).
\end{equation}
By Fekete's lemma $\frac{a_k}{km}$ converges to its supremum, which is in the interval $[0,\entropy(P)]$ (\cref{it:Fbounded}). If $m,m'\in\mathbb{N}$ and $P\in\distributions[m](V(G))\cap\distributions[m'](V(G))$ then also $P\in\distributions[mm'](V(G))$ and
\begin{equation}
\begin{split}
\lim_{k\to\infty}\frac{1}{km}\log f(\typegraph{G}{km}{P})
 & = \lim_{k\to\infty}\frac{1}{kmm'}\log f(\typegraph{G}{kmm'}{P})  \\
 & = \lim_{k\to\infty}\frac{1}{km'}\log f(\typegraph{G}{km'}{P}),
\end{split}
\end{equation}
because the sequence in the middle is a subsequence of the other two. Therefore the limit defines a function on $\distributions[\mathbb{Q}](V(G))$.

Let $P,Q\in\distributions[\mathbb{Q}](V(G))$ and $\lambda\in[0,1]\cap\mathbb{Q}$. Choose $m,n\in\mathbb{N}$ such that $m=\lambda(m+n)$, $P\in\distributions[m](V(G))$ and $Q\in\distributions[n](V(G))$. By \Cref{prop:convexcombinationbounds} we have
\begin{equation}
\begin{split}
\typegraph{G}{km}{P}\strongproduct\typegraph{G}{kn}{Q}
 & \le\typegraph{G}{k(m+n)}{\frac{mP+nQ}{m+n}}  \\
 & \le\complement{K_N}\strongproduct\typegraph{G}{km}{P}\strongproduct\typegraph{G}{kn}{Q}
\end{split}
\end{equation}
with
\begin{equation}
N\le(km+1)^{2|V(G)|}(kn+1)^{2|V(G)|}2^{k(m+n)\left[\entropy(\lambda P+(1-\lambda)Q)-\lambda\entropy(P)-(1-\lambda)\entropy(Q)\right]}
\end{equation}
Apply $f$, take the logarithm and divide by $k(m+n)$ to get
\begin{multline}
\lambda\frac{1}{km}\log f(\typegraph{G}{km}{P})+(1-\lambda)\frac{1}{kn}\log f(\typegraph{G}{kn}{Q})  \\  \le\frac{1}{k(m+n)}\log f(\typegraph{G}{k(m+n)}{\lambda P+(1-\lambda)Q})  \\  \le\lambda\frac{1}{km}\log f(\typegraph{G}{km}{P})+(1-\lambda)\frac{1}{kn}\log f(\typegraph{G}{kn}{Q})+\frac{\log N}{k(m+n)}.
\end{multline}
and take the limit $k\to\infty$:
\begin{equation}
\begin{split}
\lambda F(G,P)+(1-\lambda)F(G,Q)
 & \le F(G,\lambda P+(1-\lambda)Q)  \\
 & \le \lambda F(G,P)+(1-\lambda)F(G,Q)+\entropy(\lambda P+(1-\lambda)Q)-\lambda\entropy(P)-(1-\lambda)\entropy(Q).
\end{split}
\end{equation}
This proves \cref{it:Fconcave,it:HFconcave}.

Let $P,Q\in\distributions[\mathbb{Q}](V(G))$, $\delta=\frac{1}{2}\norm[1]{P-Q}\neq 0$ and define
\begin{align}
P' & = \delta^{-1}(P-Q)_-  \\
Q' & = \delta^{-1}(P-Q)_+
\end{align}
and $\lambda=\delta/(1+\delta)$. Then $\lambda P'+(1-\lambda)P=\lambda Q'+(1-\lambda)Q$. By concavity of $F$ and $H-F$ and using $\entropy(\lambda P+(1-\lambda)Q)-\lambda\entropy(P)-(1-\lambda)\entropy(Q)\le h(\lambda)$ we get the estimates
\begin{align}
F(G,P)-F(G,\lambda P'+(1-\lambda)P) & \le \lambda(F(G,P)-F(G,P'))  \\
F(G,\lambda Q'+(1-\lambda)Q)-F(G,Q) & \le \lambda(F(G,Q')-F(G,Q))+h(\lambda).
\end{align}
We add the two inequalities and rearrange:
\begin{equation}
\begin{split}
(1-\lambda)(F(G,P)-F(G,Q))
 & \le\lambda(F(G,Q')-F(G,P'))+h(\lambda)  \\
 & \le\lambda\log|\support Q'|+h(\lambda)  \\
 & \le\lambda\log(|V(G)|-1)+h(\lambda)
\end{split}
\end{equation}
Finally, divide by $1-\lambda$ and use the definition of $\lambda$:
\begin{equation}\label{eq:Fannestype}
F(G,P)-F(G,Q)\le \delta\log(|V(G)|-1)+h(\delta)+2\left(1-h\left(\frac{1+\delta}{2}\right)\right).
\end{equation}
The expression on the right hand side is symmetric in $P$ and $Q$, therefore it is also an upper bound on the absolute value of the left hand side, which proves \cref{it:Fcontinuous}.
\end{proof}
The probabilistic refinement of the Lov\'asz theta number was defined and studied by Marton in \cite{marton1993shannon} via a non-asymptotic formula. The probabilistic refinement of the fractional clique cover number is related to the graph entropy as $\entropy(G,P)=\log\complement{\chi}_f(\complement{G},P)$ \cite{korner1973coding}.

Clearly, $F(G,P)$ only depends on $G[\support P]$ and $P$.

We remark that the upper bound in eq. \eqref{eq:continuityestimate} is close to optimal among the expressions depending only on $|V(G)|$ and $\norm[1]{P-Q}$: if we omit the last term and specialise to $F(\complement{K_N},P)=\entropy(P)$ then it becomes sharp, see \cite[Theorem 3.8]{petz2007quantum} and \cite{audenaert2007sharp}.

The route we followed is not the only way to arrive at the probabilistic refinement. We now state its equivalence with other common definitions.
\begin{proposition}
Let $G$ be a graph and $P\in\distributions(V(G))$. Then
\begin{align}
F(G,P) & =\lim_{n\to\infty}\frac{1}{n}\log f(\typegraph{G}{n}{P_n})  \label{eq:probrefinementsequence}  \\
 & = \lim_{\epsilon\to 0}\lim_{n\to\infty}\frac{1}{n}\log f(\typegraph{G}{n}{\ball{\epsilon}{P}})  \label{eq:probrefinementtypical}  \\
 & = \lim_{n\to\infty}\min_{\substack{S\subseteq V(G)^n  \\  P^{\otimes n}(S)>c}}\frac{1}{n}\log f(G^{\strongproduct n}[S]),  \label{eq:probrefinementessential}
\end{align}
for any sequence $(P_n)_{n\in\mathbb{N}}$ such that $P_n\in\distributions[n](V(G))$ and $P_n\to P$, and any $c\in(0,1)$.
\end{proposition}
For the proof see \Cref{sec:proof}.
\begin{proposition}\label{prop:ffromF}
For any graph $G$ we have $\displaystyle f(G)=\max_{P\in\distributions(V(G))}f(G,P)$.
\end{proposition}
For the proof see \Cref{sec:proof}.

\begin{lemma}\label{prop:marginaltypeclassproduct}
Let $G,H$ be graphs and $P\in\distributions[n](V(G)\times V(H))$. Let $P_G$ and $P_H$ denote its marginals on $V(G)$ and $V(H)$, respectively. Then
\begin{equation}
\typegraph{(G\strongproduct H)}{n}{P}\le\typegraph{G}{n}{P_G}\strongproduct\typegraph{H}{n}{P_H}\le\complement{K_N}\strongproduct\typegraph{(G\strongproduct H)}{n}{P}
\end{equation}
holds in $\graphs$ for some $N\in\mathbb{N}$ satisfying $N\le(n+1)^{2|V(G)||V(H)|}2^{n\mutualinformation(G:H)_P}$.
\end{lemma}
\begin{proof}
The marginal types of any sequence in $\typeclass{n}{P}$ are $P_G$ and $P_H$, therefore $\typegraph{(G\strongproduct H)}{n}{P}$ is an induced subgraph of $\typegraph{G}{n}{P_G}\strongproduct\typegraph{H}{n}{P_H}$.

For the second inequality we apply \Cref{lem:transitiveinduced} to the graph $\typegraph{G}{n}{P_G}\strongproduct\typegraph{H}{n}{P_H}$, which comes equipped with a transitive action of $S_n\times S_n$. The upper bound on $N$ can be seen from
\begin{equation}
\begin{split}
N
 & = \left\lfloor\frac{|\typeclass{n}{P_G}\times\typeclass{n}{P_H}|}{|\typeclass{n}{P}|}\ln|\typeclass{n}{P_G}\times\typeclass{n}{P_H}|\right\rfloor+1  \\
 & \le \left\lfloor(n+1)^{|V(G)||V(H)|}2^{n\mutualinformation(G:H)_P}n\ln|V(G)||V(H)|\right\rfloor+1  \\
 & \le (n+1)^{2|V(G)||V(H)|}2^{n\mutualinformation(G:H)_P}.
\end{split}
\end{equation}
\end{proof}

\begin{proposition}\label{prop:probmultiplative}
$F$ satisfies \cref{it:probmultiplicative}.
\end{proposition}
\begin{proof}
By continuity, it is enough to verify the inequalities for distributions with rational probabilities. Let $G,H$ be graphs and $P\in\distributions[n](V(G)\times V(H))$. \Cref{prop:marginaltypeclassproduct} implies
\begin{equation}
\typegraph{(G\strongproduct H)}{kn}{P}\le\typegraph{G}{kn}{P_G}\strongproduct\typegraph{H}{kn}{P_H}\le\complement{K_N}\strongproduct\typegraph{(G\strongproduct H)}{kn}{P}
\end{equation}
where $N=(kn+1)^{2|V(G)||V(H)|}2^{kn\mutualinformation(G:H)_P}$.
Apply $f$ and $\log$ then divide by $kn$ and take the limit $k\to\infty$ to get
\begin{equation}
F(G\strongproduct H,P)\le F(G,P_G)+F(H,P_H)\le F(G\strongproduct H,P)+\mutualinformation(G:H)_P.
\end{equation}
\end{proof}

Our next goal is to study the behaviour of the probabilistic refinement under the disjoint union. We prove a more general statement because another special case will be used later and also because we believe it is interesting in itself.
\begin{definition}[{\cite[(6.1) Definition]{sabidussi1961graph}}]
Let $G$ and $(H_v)_{v\in V(G)}$ be graphs. The $G$-join $G\lexproduct(H_v)_{v\in V(G)}$ is defined as
\begin{align}
V(G\lexproduct(H_v)_{v\in V(G)}) & = \setbuild{(v,h)}{v\in V(G),h\in V(H_v)}  \\
E(G\lexproduct(H_v)_{v\in V(G)}) & = \setbuild{\{(v,h),(v',h')\}}{\{v,v'\}\in E(G)\text{ or }(v=v'\text{ and }\{h,h'\}\in E(H_v))}.
\end{align}
\end{definition}
This operation simultaneously generalizes the lexicographic product (when $H_v=H$ is independent of $v$), the join (when $G=K_2$), the disjoint union (when $G=\complement{K_2}$) and substitution \cite[\S 1]{lovasz1972normal} (when $K_v=K_1$ for all except one vertex $v\in V(G)$).

\begin{lemma}
Let $G$ and $(H_v)_{v\in V(G)}$ be graphs, $P\in\distributions[n](V(G\lexproduct(H_v)_{v\in V(G)}))$ and write it as $P(v,h)=P_G(v)P_v(h)$ with $P_G\in\distributions[n](V(G))$ and $\forall v\in V(G):P_v\in\distributions[nP_G(v)](V(H_v))$. Then
\begin{equation}
\begin{split}
\complement{K_{\alpha(\typegraph{G}{n}{P_G})}}\strongproduct\bigstrongproduct_{v\in V(G)}\typegraph{H_v}{nP_G(v)}{P_v}
 & \le\typegraph{(G\lexproduct(H_v)_{v\in V(G)})}{n}{P}  \\
 & \le \typegraph{G}{n}{P_G}\strongproduct\bigstrongproduct_{v\in V(G)}\typegraph{H_v}{nP_G(v)}{P_v}
\end{split}
\end{equation}
\end{lemma}
\begin{proof}
For the first inequality, let $S$ be an independent set of size $\alpha(\typegraph{G}{n}{P_G})$ in $\typegraph{G}{n}{P_G}$ and identify the vertex set of the left hand side with $S\times\prod_{v\in V(G)}\typeclass{nP_G(v)}{P_v}\subseteq\typeclass{n}{P}$. Then the left hand side is an induced subgraph of $\typegraph{(G\lexproduct(H_v)_{v\in V(G)})}{n}{P}$.

In the second inequality, the vertex sets of both sides are in a natural bijection and every edge of the graph in the right hand side is also an edge of $\typegraph{(G\lexproduct(H_v)_{v\in V(G)})}{n}{P}$.
\end{proof}

\begin{proposition}\label{prop:Gjoin}
Let $G$ and $(H_v)_{v\in V(G)}$ be graphs, $P\in\distributions(V(G\lexproduct(H_v)_{v\in V(G)}))$ and write it as $P(v,h)=P_G(v)P_v(h)$ with $P_G\in\distributions(V(G))$ and $\forall v\in V(G):P_v\in\distributions(V(H_v))$. Then
\begin{equation}
\begin{split}
\log\Theta(G,P_G)+\sum_{v\in V(G)}P_G(v)F(H_v,P_v)
 & \le F(G\lexproduct(H_v)_{v\in V(G)},P)  \\
 & \le F(G,P_G)+\sum_{v\in V(G)}P_G(v)F(H_v,P_v)
\end{split}
\end{equation}
where $\displaystyle\log\Theta(G,P_G)$ is the probabilistic refinement of the Shannon capacity.
\end{proposition}
\begin{proof}
Similar to the proof of \Cref{prop:probmultiplative}.
\end{proof}
In particular, if $G$ is perfect then the upper and lower bounds are the same. We will primarily be interested in two special cases:
\begin{corollary}\label{cor:joindisjointunion}
Let $H_1,H_2$ be graphs, $p\in[0,1]$, $P_1\in\distributions(V(H_1))$ and $P_2\in\distributions(V(H_2))$. Then
\begin{align}
F(H_1\sqcup H_2,pP_1\oplus(1-p)P_2) & = pF(H_1,P_1)+(1-p)F(H_2,P_2)+h(p)
\intertext{and}
F(H_1+ H_2,pP_1\oplus(1-p)P_2) & = pF(H_1,P_1)+(1-p)F(H_2,P_2).
\end{align}
In particular, $F$ satisfies \cref{it:probadditive}.
\end{corollary}
\begin{proof}
Both statements follow from \Cref{prop:Gjoin}, the first one with $G=\complement{K_2}$ and using $\log\Theta(\complement{K_2},(p,1-p))=h(p)$, while the second one with $G=K_2$ and using $\log\Theta(K_2,(p,1-p))=0$.
\end{proof}

\begin{proposition}
Let $G,H$ be graphs, $\varphi:\complement{H}\to\complement{G}$ a homomorphism and $P\in\distributions(V(H))$. Then $F(H,P)\le F(G,\varphi_*(P))$, i.e. \cref{it:probmonotone} is satisfied.
\end{proposition}
\begin{proof}
We prove the statement for $P\in\distributions[n](V(H))$, the general case follows by continuity. For any $k\in\mathbb{N}$ we have a homomorphism $\varphi^{\strongproduct kn}:\complement{H^{\strongproduct kn}}\to\complement{G^{\strongproduct kn}}$, and the image of $\typeclass{kn}{P}$ is $\typeclass{kn}{\varphi_*(P)}$. Thus $\typegraph{H}{kn}{P}\le\typegraph{G}{kn}{\varphi_*(P)}$. Apply $\frac{1}{kn}\log f$ to both sides and let $k\to\infty$.
\end{proof}
This finishes the proof of \Cref{thm:spectralpointtoprobabilistic}. Now we turn to the converse direction.
\begin{proof}[Proof of \Cref{thm:probabilistictospectralpoint}.]
$K_1\strongproduct K_1$ is isomorphic to $K_1$ and there is only one probability distribution on its vertex set, therefore $F(K_1,P)=2F(K_1,P)=0$ by \cref{it:probmultiplicative}. It follows that $f(K_1)=1$, proving \cref{it:spectralnormalised}.

We prove supermultiplicativity. Let $G,H$ be nonempty graphs, $P_G\in\distributions(V(G))$ and $P_H\in\distributions(V(H))$ such that $\log f(G)=F(G,P_G)$ and $\log f(H)=F(H,P_H)$. Using $\mutualinformation(G:H)_{P_G\otimes P_H}=0$ and \cref{it:probmultiplicative} we have
\begin{equation}
\log f(G\strongproduct H) \ge F(G\strongproduct H,P_G\otimes P_H) = F(G,P_G)+F(H,P_H)=\log f(G)+\log f(H).
\end{equation}

We prove submultiplicativity. Let $G,H$ be nonempty graphs, $P\in\distributions(V(G\strongproduct H))$ such that $\log f(G\strongproduct H)=F(G\strongproduct H,P)$. Let $P_G$ and $P_H$ be the marginals of $P$. Then
\begin{equation}
\log f(G\strongproduct H) = F(G\strongproduct H,P)\le F(G,P_G)+F(H,P_H)\le\log f(G)+\log f(H).
\end{equation}
This proves \cref{it:spectralmultiplicative}.

We prove additivity (\cref{it:spectraladditive}). Let $G,H$ be nonempty graphs, $P_G\in\distributions(V(G))$ and $P_H\in\distributions(V(H))$ and $p\in[0,1]$. Then $pP_G\oplus(1-p)P_H\in\distributions(V(G\sqcup H))$ and any distribution on the union arises in this way. Thus
\begin{equation}
\begin{split}
\log f(G\sqcup H)
 & = \max_{p,P_G,P_H}F(G\sqcup H,pP_G\oplus(1-p)P_H)  \\
 & = \max_{p,P_G,P_H}(pF(G,P_G)+(1-p)F(H,P_H)+h(p))  \\
 & = \max_{p\in[0,1]}(p\log f(G)+(1-p)\log f(H)+h(p))  \\
 & = \log (f(G)+f(H))
\end{split}
\end{equation}
by \cref{it:probadditive}.

We prove that $f$ is monotone (\cref{it:spectralmonotone}). Let $G,H$ be graphs, $\varphi:\complement{H}\to\complement{G}$ be a homomorphism and $P\in\distributions(V(H))$ such that $\log f(H)=F(H,P)$. Then
\begin{equation}
\log f(G)=\max_{Q\in\distributions(V(G))}F(G,Q)\ge F(G,\varphi_*(P))\ge F(H,P)=\log f(H)
\end{equation}
by \cref{it:probmonotone}.

We prove that $F$ is the logarithmic probabilistic refinement of $f$. Let $G$ be a graph and $P\in\distributions[m](V(G))$. $\typegraph{G}{km}{P}$ is vertex-transitive, $F(G,\cdot)$ is concave (\cref{it:probconcave}), therefore
\begin{equation}
\log f(\typegraph{G}{km}{P})=F(\typegraph{G}{km}{P},U)
\end{equation}
where $U$ denotes the uniform distribution. $U$ can be considered a distribution on $V(G)^{km}$, and then all $km$ marginals have distribution $P$. From \cref{it:probmultiplicative} follows by induction that
\begin{equation}\label{eq:probmultiplicativetypeclass}
F(G^{\strongproduct km},U)\le km F(G,P)\le F(G^{\strongproduct km},U)+km\entropy(P)-\entropy(U),
\end{equation}
which in turn implies
\begin{equation}
\begin{split}
\lim_{k\to\infty}\frac{1}{km}F(G^{\strongproduct km},U)
 & \le F(G,P)  \\
 & \le \lim_{k\to\infty}\frac{1}{km}F(G^{\strongproduct km},U)+\entropy(P)-\frac{1}{km}\entropy(U)  \\
 & \le \lim_{k\to\infty}\frac{1}{km}F(G^{\strongproduct km},U).
\end{split}
\end{equation}
\end{proof}

\section{Further properties}\label{sec:furtherproperties}

\subsection{Capacity within a fixed distribution}

Recall that the Shannon capacity $\Theta(G)$ can be expressed as the minimum of $f(G)$ as $f$ runs over the asymptotic spectrum $\Delta(\graphs)$. We prove the analogous statement for the probabilistic refinement.
\begin{proposition}\label{prop:capacitywithintype}
Let $G$ be a graph and $P\in\distributions(V(G))$. Then
\begin{equation}
\Theta(G,P)=\min_{f\in\Delta(\graphs)}f(G,P).
\end{equation}
\end{proposition}
\begin{proof}
$\alpha(H)\le f(H)$ for any graph $H$ and spectral point $f$. Therefore
\begin{equation}
\begin{split}
\Theta(G,P)
 & = \lim_{\epsilon\to 0}\limsup_{n\to\infty}\sqrt[n]{\alpha(\typegraph{G}{n}{\ball{\epsilon}{P}})}  \\
 & \le \lim_{\epsilon\to 0}\limsup_{n\to\infty}\sqrt[n]{f(\typegraph{G}{n}{\ball{\epsilon}{P}})} = f(G,P).
\end{split}
\end{equation}

Let $\epsilon>0$ and $Q\in\ball{\epsilon}{P}\cap\distributions[m](V(G))$ for some $m\in\mathbb{N}$. Then $(\typeclass{m}{Q})^k\subseteq\typeclass{km}{Q}\subseteq\typeclass{km}{\ball{\epsilon}{P}}$ for any $k\in\mathbb{N}$, therefore
\begin{equation}
\begin{split}
\limsup_{n\to\infty}\sqrt[n]{\alpha(\typegraph{G}{n}{\ball{\epsilon}{P}})}
 & \ge \limsup_{k\to\infty}\sqrt[km]{\alpha((\typegraph{G}{m}{Q})^{\strongproduct k})}  \\
 & = \sqrt[m]{\Theta(\typegraph{G}{m}{Q})}  \\
 & = \min_{f\in\Delta(\graphs)}\sqrt[m]{f(\typegraph{G}{m}{Q})}  \\
 & \ge \min_{f\in\Delta(\graphs)}f(G,Q)\sqrt[m]{2^{m\entropy(Q)}|\typeclass{m}{Q}|}  \\
 & \ge \min_{f\in\Delta(\graphs)}f(G,Q)(m+1)^{|V(G)|/m}
\end{split}
\end{equation}
by \cref{eq:probmultiplicativetypeclass,eq:typeclasssize}. The functions $Q\mapsto f(G,Q)$ form an equicontinuous family by \cref{eq:continuityestimate}, therefore as we let $m\to\infty$ and $Q\to P$, the limit of the right hand side is $\min_{f\in\Delta(\graphs)}f(G,P)$.
\end{proof}

\subsection{Subadditivity}

Graph entropy, defined in \cite{korner1973coding}, satisfies the following subadditivity property as shown by K\"orner in \cite[Lemma 1.]{korner1986fredman}:
\begin{equation}
\entropy(G\cup H,P)\le\entropy(G,P)+\entropy(H,P)
\end{equation}
for any graphs $G,H$ with common vertex set. The same inequality is true for $\mu(G,P)=\log\vartheta(\complement{G},P)$ \cite[Lemma 2.]{marton1993shannon}. Noting that $\entropy(G,P)=\log\complement{\chi}_f(\complement{G},P)$, the following can be seen as the analogous statement for the logarithmic probabilistic refinement of an arbitrary spectral point:
\begin{proposition}\label{prop:subadditive}
Let $G,H$ be graphs on the same vertex set $V(G)$. Then for any $P\in\distributions(V(G))$ we have $F(G\cap H,P)\le F(G,P)+F(H,P)$. In particular, $\entropy(P)\le F(G,P)+F(\complement{G},P)$.
\end{proposition}
\begin{proof}
Let $\Delta:V(G)\to V(G)^2$ be the diagonal map $v\mapsto(v,v)$. This is a homomorphism from $\complement{G\cap H}$ to $\complement{G\strongproduct H}$, therefore
\begin{equation}
F(G\cap H,P)\le F(G\strongproduct H,\Delta_*P)\le F(G,P)+F(H,P),
\end{equation}
where we used \cref{it:probmultiplicative,it:probmonotone} and that both marginals of $\Delta_*P$ are equal to $P$.

The second claim follows from the first one and $F(G\cap\complement{G},P)=F(\complement{K_{|V(G)|}},P)=\entropy(P)$.
\end{proof}

It follows that any spectral point $f$ satisfies $|V(G)|=2^{\entropy(U)}\le f(G,U)f(\complement{G},U)\le f(G)f(\complement{G})$, where $G$ is any nonempty graph and $U$ is the uniform distribution on $V(G)$.

\subsection{Graph join}

It is easy to see that value of the spectral points mentioned in the introduction on the join of two graphs is the maximum of their values on the terms. One consequence of \Cref{cor:joindisjointunion} is that this is true for every spectral point.
\begin{proposition}\label{prop:fjoin}
Let $G,H$ be graphs, $f\in\Delta(\graphs)$. Then $f(G+H)=\max\{f(G),f(H)\}$.
\end{proposition}
\begin{proof}
Using the second equality in \Cref{cor:joindisjointunion} we have
\begin{equation}
\begin{split}
\log f(G+H)
 & = \max_{P\in\distributions(V(G)\sqcup V(H))}F(G+H,P)  \\
 & = \max_{p\in[0,1]}\max_{\substack{P_G\in\distributions(V(G))  \\  P_H\in\distributions(V(H))}}F(G+H,pP_G\oplus(1-p)P_H)  \\
 & = \max_{p\in[0,1]}\max_{\substack{P_G\in\distributions(V(G))  \\  P_H\in\distributions(V(H))}}pF(G,P_G)+(1-p)F(H,P_H)  \\
 & = \max_{p\in[0,1]}p\log f(G)+(1-p)\log f(H)  \\
 & = \max\{\log f(G),\log f(H)\}.
\end{split}
\end{equation}
\end{proof}

\subsection{Lexicographic product}

From $G\strongproduct H\subseteq G\lexproduct H$, $\strongproduct$-multiplicativity and monotonicity follows that any spectral point is $\lexproduct$-submultiplicative. In fact, all known elements of $\Delta(\graphs)$ are multiplicative with respect to the lexicographic product. For $\vartheta$ this follows from \cite[Section 21.]{knuth1994sandwich}, for $\xi_f$ it is proved in \cite[Theorem 27.]{cubitt2014bounds}, while for $\complement{\chi}_f$ it is an immediate consequence of \cite[Corollary 3.4.5]{scheinerman1997fractional} and $\complement{G\lexproduct H}=\complement{G}\lexproduct\complement{H}$. Although \cite[Theorem 3.]{bukh2018fractional} only states $\strongproduct$-multiplicativity of $\mathcal{H}^{\mathbb{F}}_f$, their proof actually shows $\lexproduct$-multiplicativity as well.

However, it is not clear whether this property, henceforth referred to as
\begin{enumerate}[({L})]
\item\label{it:lexicographicproduct} $f(G\lexproduct H)=f(G)f(H)$ for any pair of graphs $G,H$.
\end{enumerate}
is satisfied by every spectral point. In terms of the probabilistic refinement, we have the following characterization:
\begin{proposition}\label{prop:lexmultiplicative}
Let $f\in\Delta(\graphs)$ be a spectral point and $F$ the logarithm of its probabilistic refinement. The following are equivalent
\begin{enumerate}[(i)]
\item\label{it:propertyL} $f$ satisfies \cref{it:lexicographicproduct}
\item\label{it:probGjoineq} $F(G\lexproduct(H_v)_{v\in V(G)},P)=F(G,P_G)+\sum_{v\in V(G)}P_G(v)F(H_v,P_v)$ for any graphs $G$ and $(H_v)_{v\in V(G)}$, and probability distribution $P\in\distributions(V(G\lexproduct(H_v)_{v\in V(G)}))$, where $P(v,h)=P_G(v)P_v(h)$ with $P_G\in\distributions(V(G))$ and $\forall v\in V(G):P_v\in\distributions(V(H_v))$.
\item\label{it:problexproducteq} $F(G\lexproduct H,P)=F(G,P_G)+\sum_{v\in V(G)}P_G(v)F(H,P_v)$ for any graphs $G$ and $H$, and probability distribution $P\in\distributions(V(G\lexproduct H))$, where $P(v,h)=P_G(v)P_v(h)$ with $P_G\in\distributions(V(G))$ and $\forall v\in V(G):P_v\in\distributions(V(H))$.
\end{enumerate}
\end{proposition}
\begin{proof}
\ref{it:propertyL}$\implies$\ref{it:probGjoineq}: We use the following strengthening of the lower bound from \Cref{prop:Gjoin}:
\begin{equation}
\typegraph{G}{n}{P_G}\lexproduct\bigstrongproduct_{v\in V(G)}\typegraph{H_v}{nP_G(v)}{P_v}
\le\typegraph{(G\lexproduct(H_v)_{v\in V(G)})}{n}{P}.
\end{equation}
Apply $\frac{1}{n}\log f$ and let $n\to\infty$. Since $f$ is assumed to be $\lexproduct$-multiplicative, the obtained lower bound on $F(G\lexproduct(H_v)_{v\in V(G)},P)$ matches the upper bound from \Cref{cor:joindisjointunion}.

\ref{it:probGjoineq}$\implies$\ref{it:problexproducteq}: The $G$-join specializes to the lexicographic product when $H=H_v$ for all $v\in V(G)$.

\ref{it:problexproducteq}$\implies$\ref{it:propertyL}: We can perform maximization over $P$ by first maximizing over each $P_v$ separately with fixed $P_G$ and then over $P_G$.
\begin{equation}
\begin{split}
\log f(G\lexproduct H)
 & = \max_{P\in\distributions(V(G\lexproduct H))}F(G\lexproduct H,P)  \\
 & = \max_{P\in\distributions(V(G\lexproduct H))}F(G,P_G)+\sum_{v\in V(G)}P_G(v)F(H,P_v)  \\
 & = \max_{P_G\in\distributions(V(G))}F(G,P_G)+\sum_{v\in V(G)}P_G(v)\max_{P_v\in\distributions(V(H))}F(H,P_v)  \\
 & = \max_{P_G\in\distributions(V(G))}F(G,P_G)+\sum_{v\in V(G)}P_G(v)\log f(H)  \\
 & = \max_{P_G\in\distributions(V(G))}F(G,P_G)+\log f(H)  \\
 & = \log f(G)+\log f(H).
\end{split}
\end{equation}
\end{proof}
A consequence of \Cref{prop:lexmultiplicative} is that the subset of spectral points satisfying \cref{it:lexicographicproduct} is convex (with respect to the convex combination taken on the logarithmic probabilistic refinements).

We do not know if there are any spectral points for which \cref{it:lexicographicproduct} is false. If it turned out that \cref{it:lexicographicproduct} is true for every $f\in\Delta(\graphs)$, then in particular $\Theta(G\strongproduct H)=\Theta(G\lexproduct H)$ would be true for any graphs $G,H$.

\subsection{Logarithmic convexity}

The set of functions satisfying \cref{it:probconcave,it:probmultiplicative,it:probadditive,it:probmonotone} is convex which, by virtue of the bijection proved in \Cref{sec:probabilisticrefinement}, endows $\Delta(\graphs)$ with a convex structure. For any $(G,P)$ the function $f\mapsto \log f(G,P)$ is affine, therefore the logarithmic evaluation map
\begin{equation}
\log\ev_G:f\mapsto\log f(G)=\max_{P\in\distributions(V(G))}\log f(G,P)
\end{equation}
is a convex function on $\Delta(\graphs)$. We have the following interesting consequence.
\begin{proposition}\label{prop:incomparable}
Let $f_0,f_1\in\Delta(\graphs)$ be incomparable in the sense that there are graphs $G_0,G_1$ such that
\begin{align}
f_0(G_0) & < f_1(G_0)  \\
f_1(G_1) & < f_0(G_1).
\end{align}
Then there is an $f\in\Delta(\graphs)$ and a graph $G$ such that $f(G)<\min\{f_0(G),f_1(G)\}$, i.e. $f$ gives a better bound on the Shannon capacity of $G$ than either of $f_0,f_1$.
\end{proposition}
\begin{proof}
Let $F_i(G,P)=\log f_i(G,P)$ (for $i=0,1$) be the (logarithmic) probabilistic refinements and for $\lambda\in[0,1]$ set $F_\lambda(G,P)=(1-\lambda)F_0(G,P)+\lambda F_1(G,P)$ and $f_\lambda(G):=\max_{P\in\distributions(V(G))}2^{F(G,P)}$. Then $f_\lambda\in\Delta(\graphs)$ and $\lambda\mapsto \log f_\lambda(G)$ is convex for any $G$. In particular, $f_{1/2}(G)\le\sqrt{f_0(G)f_1(G)}$ for any graph $G$.

Choose graphs $G_0$ and $G_1$ as in the condition. First we construct new graphs $G_0'$ and $G_1'$ such that $f_0(G_0')\approx f_1(G_1')<f_0(G_1')\approx f_1(G_0)$. To this end, for $\epsilon\in(0,\frac{1}{7})$ choose $n_0,n_1\in\mathbb{N}$ such that
\begin{equation}
\left|\frac{n_0}{n_1}\frac{\log f_1(G_0)-\log f_0(G_0)}{\log f_0(G_1)-\log f_1(G_1)}-1\right|<\epsilon.
\end{equation}
Assume without loss of generality that $f_0(G_0)^{n_0}\le f_1(G_1)^{n_1}$, and choose $k,m\in\mathbb{N}$ such that
\begin{equation}
\left|n_0\log f_0(G_0)+\frac{1}{k}\log m-n_1\log f_1(G_1)\right|<\epsilon n_1(\log f_0(G_1)-\log f_1(G_1)),
\end{equation}
and let $G_0'=G_0^{\strongproduct kn_0}\strongproduct\complement{K_m}$, $G_1'=G_1^{\strongproduct kn_1}$, $G=G_0'+G_1'$. Introducing
\begin{align}
a & := kn_1\log f_1(G_1)  \\
d & := kn_1(\log f_0(G_1)-\log f_1(G_1)),
\end{align}
the choices ensure the following estimates:
\begin{align}
\log f_0(G_0') & = kn_0\log f_0(G_0)+\log m \in (a-\epsilon d,a+\epsilon d)  \\
\log f_1(G_0') & = kn_0\log f_1(G_0)+\log m \in (a+(1-2\epsilon)d,a+(1+2\epsilon) d)  \\
\log f_0(G_1') & = kn_1\log f_0(G_1) = a+d  \\
\log f_1(G_1') & = a.
\end{align}
It follows that
\begin{align}
\log f_0(G) & = \max\{\log f_0(G_0'),\log f_0(G_1')\} = a+d  \\
\log f_1(G) & = \max\{\log f_1(G_0'),\log f_1(G_1')\} > a+(1-2\epsilon)d  \\
\begin{split}
\log f_{1/2}(G)
 & = \max\{\log f_{1/2}(G_0'),\log f_{1/2}(G_1')\}  \\
 & \le\frac{1}{2}\max\{\log f_0(G_0')+\log f_1(G_0'),\log f_0(G_1')+\log f_1(G_1')\}  \\
 & <\frac{1}{2}\max\{(a+\epsilon d)+(a+(1+2\epsilon)d),(a+d)+a\}=a+\frac{1+3\epsilon}{2}d,
\end{split}
\end{align}
therefore $f_{1/2}(G)<\min\{f_0(G),f_1(G)\}$.
\end{proof}
We remark that $f_{1/2}$ also gives some improvement over both $f_0$ and $f_1$ for the graph $G_0'\sqcup G_1'$ when $\epsilon$ is small. Using the upper bound $f_\lambda(G)\le f_0(G)^\lambda f_1(G)^{1-\lambda}$ we also get a new proof of \cite[Corollary 2.]{hu2018bound}.

To illustrate \Cref{prop:incomparable} with a concrete example, let $J^p_n$ denote the graph on $\binom{[n]}{p+1}$ where the subsets $X$ and $Y$ are adjacent iff $p$ does not divide $|X\cap Y|$ \cite{haemers1978upper,bukh2018fractional}, and choose $G_0=J^2_{12}$, and $G_1=C_5$ (five-cycle). Consider the Lov\'asz number and the fractional Haemers bound over $\mathbb{F}_2$, which evaluate to $\vartheta(G_0)=\frac{260}{11}$, $\mathcal{H}^{\mathbb{F}_2}_f(G_0)=12$ \cite[Lemma 12.]{bukh2018fractional}, $\vartheta(G_1)=\sqrt{5}$ \cite{lovasz1979shannon} and $\mathcal{H}^{\mathbb{F}_2}_f(G_1)=\frac{5}{2}$ \cite[Proposition 4.]{bukh2018fractional}. With $k=n_0=1$, $m=10$, $n_1=6$ we get $G_0'=J^2_{12}\strongproduct\complement{K_{10}}$ and $G_1'=C_5^{\strongproduct 6}$ and for $G=G_0'+G_1'$ the bounds
\begin{align}
\vartheta(G) & = \frac{2600}{11} > 236  \\
\mathcal{H}^{\mathbb{F}_2}_f(G) & = \frac{15625}{64} > 244  \\
f_{1/2}(G) & \le \frac{625\sqrt{5}}{8} < 175.
\end{align}

\subsection{Complementary function}

The complementary graph entropy
\begin{equation}
\entropy^*(G,P)=\lim_{\epsilon\to 0}\limsup_{n\to\infty}\frac{1}{n}\log\chi(\typegraph{G}{n}{\ball{\epsilon}{P}})
\end{equation}
was introduced by K\"orner and Longo in \cite{korner1973two}, and can be seen as the probabilistic refinement of the Witsenhausen rate \cite{witsenhausen1976zero}
\begin{equation}
R(G)=\lim_{n\to\infty}\frac{1}{n}\log\chi(G^{\strongproduct n}).
\end{equation}
Complementary graph entropy is related to the Shannon capacity as $\entropy^*(G,P)=\entropy(P)-C(G,P)$ \cite{marton1993shannon}. This expression motivates the following definition.
\begin{definition}
Let $f\in\Delta(\graphs)$ and $F$ its logarithmic probabilistic refinement. We define
\begin{align}
F^*(G,P) & = \entropy(P)-F(\complement{G},P)
\intertext{and the complementary function}
f^*(G) & = \begin{cases}
\max_{P\in\distributions(V(G))}2^{F^*(G,P)} & \text{if $G$ is nonempty}  \\
0 & \text{if $G$ is empty.}
\end{cases}
\end{align}
\end{definition}
We choose to include the complementation in the definition to ease comparison with the original function: by \Cref{prop:subadditive} we have $F^*(G,P)\le F(G,P)$. With this definition
\begin{align}
\entropy(G,P) & = \max_{f\in\Delta(\graphs)}F(\complement{G},P)  \label{eq:graphentropy}  \\
\entropy^*(G,P) & = \max_{f\in\Delta(\graphs)}F^*(\complement{G},P).  \label{eq:complementarygraphentropy}
\end{align}
The non-probabilistic version of \cref{eq:graphentropy} was noted in \cite[Remark 1.2.]{zuiddam2018asymptotic}, whereas \cref{eq:complementarygraphentropy} follows from $\entropy^*(G,P)=\entropy(P)-C(G,P)$ and \Cref{prop:capacitywithintype}.

\begin{proposition}\label{prop:Fstarproperties}
\leavevmode
\begin{enumerate}[(i)]
\item Let $G,H$ be graphs, $p\in[0,1]$, $P_G\in\distributions(V(G))$ and $P_H\in\distributions(V(H))$. Then
\begin{align}
F^*(G\sqcup H,pP_G\oplus(1-p)P_H) & = pF^*(G,P_G)+(1-p)F^*(H,P_H)+h(p)  \\
F^*(G+ H,pP_G\oplus(1-p)P_H) & = pF^*(G,P_G)+(1-p)F^*(H,P_H).
\end{align}
\item Let $G,H$ be graphs, $P\in\distributions(V(G\costrongproduct H))$ and let $P_G$, $P_H$ denote its marginals on $V(G)$ and $V(H)$. Then
\begin{equation}
F^*(G\costrongproduct H,P)\le F^*(G,P_G)+F^*(H,P_H)\le F^*(G\costrongproduct H,P)+\mutualinformation(G:H)_P
\end{equation}
\end{enumerate}
\end{proposition}
The proof is a straightforward calculation (see \Cref{sec:proof}).

\begin{theorem}
$f^*(\complement{K_n})=n$ and for any pair of graphs $G,H$ we have
\begin{align}
f^*(G\costrongproduct H) & = f^*(G)f^*(H)  \\
f^*(G\sqcup H) & = f^*(G)+f^*(H)  \\
f^*(G+ H) & = \max\{f^*(G),f^*(H)\}.
\end{align}

Suppose that $f$ satisfies \cref{it:lexicographicproduct}. Then $f^*$ also satisfies \cref{it:lexicographicproduct} and in addition, $H\le G$ implies $f^*(H)\le f(G)$ for any graphs $G,H$.
\end{theorem}
\begin{proof}
The first part is identical to the proof of parts of \Cref{thm:probabilistictospectralpoint} and \Cref{prop:fjoin}.

Suppose that $f$ is $\lexproduct$-multiplicative and let $G,H$ be nonempty graphs. By \Cref{prop:lexmultiplicative},
\begin{equation}
\begin{split}
\log f^*(G\lexproduct H)
 & = \max_{P\in V(G\lexproduct H)}F^*(G\lexproduct H,P)  \\
 & = \max_{P\in V(G\lexproduct H)}\entropy(P)-F(\complement{G}\lexproduct\complement{H},P)  \\
 & = \max_{P\in V(G)}\entropy(P_G)-F(\complement{G},P_G)+\sum_{v\in V(G)}P_G(v)\max_{P_v\in V(H)}\left(\entropy(P_v)-F(\complement{H},P_v)\right)  \\
 & = \max_{P\in V(G)}F^*(G,P_G)+\sum_{v\in V(G)}P_G(v)\max_{P_v\in V(H)}F^*(H,P_v)  \\
 & = \max_{P\in V(G)}F^*(G,P_G)+\sum_{v\in V(G)}P_G(v)\log f^*(H)  \\
 & = \log f^*(G)+\log f^*(H).
\end{split}
\end{equation}

Suppose again that $f$ is $\lexproduct$-multiplicative and let $G,H$ be nonempty graphs such that $H\le G$. Let $\varphi:\complement{H}\to\complement{G}$ be a homomorphism. Define $H'$ as the graph with $V(H')=V(H)$ and $\{u,v\}\in E(H')$ iff $\{\varphi(u),\varphi(v)\}\in E(G)$ or $\varphi(u)=\varphi(v)$ (but $u\neq v$). Then $H'\subseteq H$, therefore
\begin{equation}
\log f^*(H')=\max_{P\in\distributions(V(H'))}\entropy(P)-F(\complement{H'},P)\ge\max_{P\in\distributions(V(H))}\entropy(P)-F(\complement{H},P)=\log f^*(H).
\end{equation}
On the other hand, $H'$ is the $G$-join of the family of complete graphs $(K_{\varphi^{-1}(v)})_{v\in V(G)}$, therefore
\begin{equation}
\begin{split}
\log f^*(H')
 & = \max_{P\in\distributions(V(H'))}\entropy(P)-F(\complement{H'},P)  \\
 & = \max_{P\in\distributions(V(H'))}\entropy(P)-F(\complement{G}\lexproduct(\complement{K_{\varphi^{-1}(v)}})_{v\in V(G)},P)  \\
 & = \max_{P_G\in\distributions(\varphi(V(H)))}\entropy(P_G)-F(\complement{G},P_G)+\sum_{v\in\varphi(V(H))}P_G(v)\max_{P_v\in\distributions(\varphi^{-1}(v))}\left(\entropy(P_v)-F(\complement{K_{\varphi^{-1}(v)}},P_v)\right)  \\
 & = \max_{P_G\in\distributions(\varphi(V(H)))}\entropy(P_G)-F(\complement{G},P_G)  \\
 & \le \log f^*(G)
\end{split}
\end{equation}
by \Cref{prop:lexmultiplicative} and using $F(\complement{K_{\varphi^{-1}(v)}},P_v)=\entropy(P_v)$.
\end{proof}
For example, the Lov\'asz theta function $\vartheta$ satisfies $\log\vartheta(G,P)+\log\vartheta(\complement{G},P)=\entropy(P)$ by \cite[Theorem 2.]{marton1993shannon}, which implies $\vartheta^*(G)=\vartheta(G)$ and the well-known property that $\vartheta$ is also multiplicative under the costrong product.

The complementary function for $\complement{\chi}_f$ is the independence number $\alpha$ (this is a consequence of the equality $\complement{\chi_f}(\complement{H})\alpha(H)=|V(H)|$, valid for vertex-transitive $H$). One also verifies that
\begin{align}
\min_{f\in\Delta(\graphs)}f^*(G) & = \alpha(G)  \\
\max_{f\in\Delta(\graphs)}f^*(G) & = 2^{R(\complement{G})}.
\end{align}

\subsection{Splitting}

A graph $G$ is called strongly splitting if $\entropy(P)=\entropy(G,P)+\entropy(\complement{G},P)$ for every $P\in\distributions(V(G))$. In \cite{korner1988graphs} K\"orner and Marton conjectured that a graph is strong splitting if and only if it is perfect, which was proved in \cite[1.2 Theorem]{csiszar1990entropy}. This motivates the following definition.
\begin{definition}
A graph $G$ strongly splits $F$ if $\entropy(P)=F(G,P)+F(\complement{G},P)$ for all $P\in\distributions(V(G))$.
\end{definition}
For example, perfect graphs strongly split $F$ for every spectral point $f$, while every graph strongly splits $\log\vartheta$. Clearly, if $G$ strongly splits $F$ then $\complement{G}$ as well as any induced subgraph of $G$ strongly splits $F$. In terms of the complementary function $F^*$ we may say that $G$ strongly splits $F$ if $F(G,\cdot)=F^*(G,\cdot)$.
\begin{proposition}
If $G$ and $H$ both strongly split $F$ then so do $G\sqcup H$ and $G\lexproduct H$.
\end{proposition}
\begin{proof}
Let $P_G\in\distributions(V(G))$, $P_H\in\distributions(V(H))$ and $p\in[0,1]$. Then
\begin{equation}
\begin{split}
& F(G\sqcup H,pP_G\oplus(1-p)P_H)+F(\complement{G\sqcup H},pP_G\oplus(1-p)P_H)  \\
 & = F(G\sqcup H,pP_G\oplus(1-p)P_H)+F(\complement{G}+\complement{H},pP_G\oplus(1-p)P_H)  \\
 & = pF(G,P_G)+(1-p)F(H,P_H)+h(p)+pF(\complement{G},P_G)+(1-p)F(\complement{H},P_H)  \\
 & = p\entropy(P_G)+(1-p)\entropy(P_H)+h(p)=\entropy(pP_G\oplus(1-p)P_H).
\end{split}
\end{equation}

For the lexicographic product we use the upper bound from \Cref{prop:Gjoin} and \Cref{prop:subadditive}. Let $P\in\distributions(V(G\lexproduct H))$, $P_G$ its marginal on $V(G)$ and for $v\in V(G)$ let $P_v$ be the conditional distribution such that $P(v,h)=P_G(v)P_v(h)$. Then
\begin{equation}
\begin{split}
\entropy(P)
 & \le F(G\lexproduct H,P)+F(\complement{G\lexproduct H},P)  \\
 & = F(G\lexproduct H,P)+F(\complement{G}\lexproduct\complement{H},P)  \\
 & \le F(G,P_G)+\sum_{v\in V(G)}P_G(v)F(H,P_v)+F(\complement{G},P_G)+\sum_{v\in V(G)}P_G(v)F(\complement{H},P_v)  \\
 & = \entropy(P_G)+\sum_{v\in V(G)}P_G(v)\entropy(P_v)=\entropy(P).
\end{split}
\end{equation}
\end{proof}
It also follows that the upper bound in \Cref{prop:Gjoin} is an equality in this case.

\section{Convex corners}\label{sec:convexcorners}

The logarithm of the probabilistic refinement of the fractional clique cover number is equal to the graph entropy of the complement:
\begin{equation}
\begin{split}
\log\complement{\chi}_f(G,P)
 & = \lim_{\epsilon\to0}\lim_{n\to\infty}\frac{1}{n}\log\complement{\chi}_f(G^{\strongproduct n}[\typeclass{n}{\ball{\epsilon}{P}}])  \\
 & = \lim_{\epsilon\to0}\lim_{n\to\infty}\frac{1}{n}\log\chi_f(\complement{G}^{\costrongproduct n}[\typeclass{n}{\ball{\epsilon}{P}}])  \\
 & = \entropy(\complement{G},P).
\end{split}
\end{equation}
In \cite[3.1 Lemma]{csiszar1990entropy}, graph entropy was expressed as a special case of a new entropy concept, as the entropy function with respect to the vertex packing polytope. In this section we show that the logarithmic probabilistic refinement of every spectral point on a fixed graph is the entropy function of a convex corner and provide a characterization of convex corner-valued graph parameters arising in this way.

\begin{definition}[\cite{fulkerson1971blocking,grotschel1986relaxations}]
Let $\mathcal{X}$ be a finite set. A convex corner on $\mathcal{X}$ is a subset $A\subseteq\mathbb{R}_{\ge 0}^{\mathcal{X}}$ which is compact, convex, has nonempty interior and for any $a\in A$ and $0\le a'\le a$ (componentwise) satisfies $a'\in A$. The antiblocker of $A$ is the set
\begin{equation}
A^*=\setbuild{b\in\mathbb{R}_{\ge 0}^{\mathcal{X}}}{\forall a\in A:\langle a,b\rangle\le 1},
\end{equation}
where $\langle\cdot,\cdot\rangle$ is the standard inner product.
\end{definition}
The antiblocker operation is an involution, i.e. $(A^*)^*=A$ for any convex corner $A$.
\begin{definition}[\cite{csiszar1990entropy}]
The entropy of $P\in\distributions(\mathcal{X})$ with respect to $A$ is
\begin{equation}
\entropy_A(P)=\min_{a\in A}-\sum_{x\in\mathcal{X}}P(x)\log a_x,
\end{equation}
where $\log 0$ is interpreted as $-\infty$.
\end{definition}
For example, the unit corner $S=\setbuild{\mathbb{R}_{\ge 0}^{\mathcal{X}}}{\sum_{x\in\mathcal{X}}a_x\le 1}$ is a convex corner and $\entropy_S(P)=\entropy(P)$ is the Shannon entropy, whereas $\entropy_{[0,1]^{\mathcal{X}}}(P)=0$. Examples of convex corners arising in graph theory are the vertex packing polytope $\vertexpackingpolytope(G)\subseteq\mathbb{R}_{\ge0}^{V(G)}$, which is the convex hull of the characteristic vectors of independent sets of $G$, and the theta body $\thetabody(G)=\setbuild{(\langle u_v,c\rangle^2)_{v\in V(G)}}{(u,c)\in T(G)}\subseteq\mathbb{R}_{\ge0}^{V(G)}$, where $T(G)$ is the set of orthonormal representations with handle, i.e. $(u,c)\in T(G)$ if $u:V(G)\to\mathbb{R}^n$ and $c\in\mathbb{R}^n$ for some $n$, $\norm{u_v}=\norm{c}=1$ for all $v\in V(G)$ and $\langle u_v,u_w\rangle=0$ whenever $\{v,w\}\in E(G)$ \cite{grotschel1986relaxations}.

The entropy with respect to $A$ is the minimum of a set of affine functions, therefore it is a concave function of the distribution. The entropy and the antiblocker are related as $\entropy(P)=\entropy_A(P)+\entropy_{A^*}(P)$ \cite[2.4 Corollary]{csiszar1990entropy}.

\begin{lemma}[{\cite[2.1 Lemma]{csiszar1990entropy}}]\label{lem:subcorner}
For two convex corners $A,B\subseteq\mathbb{R}_{\ge 0}^{\mathcal{X}}$, we have $\entropy_A(P)\ge\entropy_B(P)$ for all $P\in\distributions(\mathcal{X})$ if and only if $A\subseteq B$.
\end{lemma}
By \Cref{lem:subcorner}, a convex corner can be recovered from its entropy function. This is true essentially because the entropy is also the minimum of a linear objective function on $\Lambda(A\cap\mathbb{R}_{>0}^{\mathcal{X}})$ where $\Lambda(a)=(-\log a_x)_{x\in\mathcal{X}}$. However, it is not immediately clear when is a function on probability distributions the entropy function of a convex corner. This question is answered by the following proposition.
\begin{proposition}\label{prop:entropycharacterization}
Let $\mathcal{X}$ be a finite set and $F:\distributions(\mathcal{X})\to\mathbb{R}$. The following conditions are equivalent:
\begin{enumerate}[(i)]
\item\label{it:convexcornerentropy} $F=\entropy_A$ for some convex corner $A$ on $\mathcal{X}$
\item\label{it:concaveHminusconcave} $F$ and $\entropy-F$ are concave.
\end{enumerate}
\end{proposition}
\begin{proof}
We prove that \ref{it:convexcornerentropy} implies \ref{it:concaveHminusconcave}. As remarked above, $\entropy_A$ is concave. $\entropy-\entropy_A$ is equal to $\entropy_{A^*}$, hence also concave.

We prove that \ref{it:concaveHminusconcave} implies \ref{it:convexcornerentropy}. Define
\begin{equation}
L=\setbuild{l\in\mathbb{R}^{\mathcal{X}}}{\forall P\in\distributions(\mathcal{X}):\langle P,l\rangle\ge F(P)}.
\end{equation}
$L$ is closed, convex and $l'\ge l\in L$ implies $l'\in L$. Let $\delta_x$ denote the probability measure concentrated at $x$. By concavity of $F$ and $\entropy-F$ we get the estimates
\begin{equation}
\min_{x\in\mathcal{X}}F(\delta_x)\le F(P)\le\max_{x\in\mathcal{X}}F(\delta_x)+\log|\mathcal{X}|,
\end{equation}
which in particular implies $[\max_{x\in\mathcal{X}}F(\delta_x)+\log|\mathcal{X}|,\infty)\subseteq L\subseteq[\min_{x\in\mathcal{X}}F(\delta_x),\infty)^{\mathcal{X}}$.
By linear programming duality, $F(P)=\min_{l\in L}\langle P,l\rangle$, i.e. $F=\entropy_A$ with
\begin{equation}
A=\overline{\Lambda^{-1}(L)}=\overline{\setbuild{a\in\mathbb{R}_{>0}^{\mathcal{X}}}{\forall P\in\distributions(\mathcal{X}):-\sum_{x\in\mathcal{X}}P(x)\log a_x\ge F(P)}}.
\end{equation}
From the properties of $L$ we see that $A$ is compact, has nonempty interior and $0\le a'\le a\in A$ implies $a'\in A$. We need to show that $A$ is convex.

Let $a,b\in A\cap\mathbb{R}_{>0}^{\mathcal{X}}$ and $\lambda\in(0,1)$. Introduce
\begin{equation}
q_x=\frac{\lambda a_x}{\lambda a_x+(1-\lambda)b_x}
\end{equation}
so that $\log(\lambda a_x+(1-\lambda)b_x)=q_x\log\lambda a_x+(1-q_x)\log(1-\lambda)b_x+h(q_x)$. For an arbitrary $P\in\distributions(\mathcal{X})$ let
\begin{align}
q & = \sum_{x\in\mathcal{X}}P(x)q_x  \\
Q(x) & = \frac{P(x)q_x}{q}  \\
R(x) & = \frac{P(x)(1-q_x)}{1-q}.
\end{align}
With these definitions $q\in(0,1)$, $Q,R\in\distributions(\mathcal{X})$ and $P=qQ+(1-q)R$ and
\begin{equation}
h(q)+q\entropy(Q)+(1-q)\entropy(R)=\entropy(P)+\sum_{x\in\mathcal{X}}P(x)h(q_x).
\end{equation}
Using this identity we can write
\begin{equation}
\begin{split}
& \langle P,\Lambda(\lambda a+(1-\lambda)b)\rangle  \\
 & = -\sum_{x\in\mathcal{X}}P(x)\left[q_x\log\lambda a_x+(1-q_x)\log(1-\lambda) b_x+h(q_x)\right]  \\
 & = -q\log\lambda+q\langle Q,\Lambda(a)\rangle-(1-q)\log(1-\lambda)+(1-q)\langle R,\Lambda(b)\rangle-\sum_{x\in\mathcal{X}}P(x)h(q_x)  \\
 & \ge -q\log\lambda+qF(Q)-(1-q)\log(1-\lambda)+(1-q)F(R)-\sum_{x\in\mathcal{X}}P(x)h(q_x)  \\
 & \ge F(P)-\entropy(P)+q\entropy(Q)+(1-q)\entropy(R)-q\log\lambda-(1-q)\log(1-\lambda)-\sum_{x\in\mathcal{X}}P(x)h(q_x)  \\
 & = F(P)+d(q\|\lambda)\ge F(P).
\end{split}
\end{equation}
The first inequality follows from $a,b\in A$, the second one is concavity of $\entropy-F$, while the last one is true because the omitted term is a relative entropy. Thus $\lambda a+(1-\lambda)b\in A$ so $A$ is convex, because it is the closure of the convex set $A\cap\mathbb{R}_{>0}^{\mathcal{X}}$.
\end{proof}
\Cref{prop:entropycharacterization} and \ref{it:Fconcave} and \ref{it:HFconcave} in \Cref{prop:probabilisticrefinement} imply that the probabilistic refinement of any spectral point $f$ may be encoded in a function $C_f$ mapping graphs to convex corners on their vertex sets. For example, the fractional clique cover number corresponds to $C_{\complement{\chi}_f}(G)=\vertexpackingpolytope(\complement{G})$ by \cite[3.1 Lemma]{csiszar1990entropy}, whereas the Lov\'asz number corresponds to $C_\vartheta(G)=\thetabody(\complement{G})$ \cite{marton1993shannon}.

Our next goal is to translate \cref{it:probmultiplicative,it:probadditive,it:probmonotone} into conditions on $C_f$. The first ingredient is an extension of \cite[5.1 Theorem, (i)]{csiszar1990entropy}, which is about the entropy of product distributions with respect to the tensor product of convex corners.
\begin{definition}[{\cite[Section 5.]{csiszar1990entropy}}]
Let $A$ be a convex corner on $\mathcal{X}$ and $B$ a convex corner on $\mathcal{Y}$. Their tensor product $A\otimes B$ is the smallest convex corner on $\mathcal{X}\times\mathcal{Y}$ containing the tensor products $a\otimes b$ for all $a\in A$ and $b\in B$.
\end{definition}

\begin{proposition}\label{prop:cornerproduct}
Let $A$ and $B$ be convex corners on $\mathcal{X}$ and $\mathcal{Y}$ and $C$ a convex corner on $\mathcal{X}\times\mathcal{Y}$. The following are equivalent
\begin{enumerate}[(i)]
\item\label{it:cornerbetweentensorproducts} $A\otimes B\subseteq C\subseteq(A^*\otimes B^*)^*$
\item\label{it:cornermultiplicativebounds} $\entropy_C(P)\le\entropy_A(P_X)+\entropy_B(P_Y)\le\entropy_C(P)+\mutualinformation(X:Y)_P$ for every $P\in\distributions(\mathcal{X}\times\mathcal{Y})$.
\end{enumerate}
\end{proposition}
\begin{proof}
We prove \ref{it:cornerbetweentensorproducts}$\implies$\ref{it:cornermultiplicativebounds}. First note that for any $P\in\distributions(\mathcal{X}\times\mathcal{Y})$ we have the inequality
\begin{equation}
\begin{split}
\entropy_{A\otimes B}(P)
 & \le \min_{\substack{a\in A  \\  b\in B}}-\sum_{\substack{x\in\mathcal{X}  \\  y\in\mathcal{Y}}}P(x,y)\log a_xb_y  \\
 & = \min_{\substack{a\in A  \\  b\in B}}-\sum_{x\in\mathcal{X}}P_X(x)\log a_x-\sum_{y\in\mathcal{Y}}P_Y(y)\log b_y  \\
 & = \entropy_A(P_X)+\entropy_B(P_Y),
\end{split}
\end{equation}
which implies
\begin{equation}
\begin{split}
\entropy_{(A^*\otimes B^*)^*}(P)
 & = \entropy(P)-\entropy_{A^*\otimes B^*}(P)  \\
 & \ge \entropy(P)-\entropy_{A^*}(P_X)-\entropy_{B^*}(P_Y)  \\
 & = \entropy(P)-\entropy(P_X)+\entropy_A(P_X)-\entropy(P_Y)+\entropy_B(P_Y)  \\
 & = \entropy_A(P_X)+\entropy_B(P_Y)-\mutualinformation(X:Y)_P.
\end{split}
\end{equation}
Let $A\otimes B\subseteq C\subseteq(A^*\otimes B^*)^*$. Then for any $P\in\distributions(\mathcal{X}\times\mathcal{Y})$ we have
\begin{equation}
\begin{split}
\entropy_A(P_X)+\entropy_B(P_Y)-\mutualinformation(X:Y)_P
 & \le\entropy_{(A^*\otimes B^*)^*}(P)  \\
 & \le\entropy_C(P)  \\
 & \le\entropy_{A\otimes B}(P)  \\
 & \le\entropy_A(P_X)+\entropy_B(P_Y).
\end{split}
\end{equation}

We prove \ref{it:cornermultiplicativebounds}$\implies$\ref{it:cornerbetweentensorproducts}. We use that
\begin{equation}
C=\overline{\setbuild{c\in\mathbb{R}_{>0}^{\mathcal{X}\times\mathcal{Y}}}{\forall P\in\distributions(\mathcal{X}\times\mathcal{Y}):\langle P,\Lambda(c)\rangle\ge\entropy_C(P)}}.
\end{equation}
Suppose that for all $P\in\distributions(\mathcal{X}\times\mathcal{Y})$ the inequality $\entropy_C(P)\le\entropy_A(P_X)+\entropy_B(P_Y)$ is satisfied. Then for any $a\in A\cap\mathbb{R}_{>0}^{\mathcal{X}}$. $b\in B\cap\mathbb{R}_{>0}^{\mathcal{Y}}$ and $P\in\distributions(\mathcal{X}\times\mathcal{Y})$ we have
\begin{equation}
\begin{split}
\langle P,\Lambda(a\otimes b)\rangle
 & = -\sum_{\substack{x\in\mathcal{X}  \\  y\in\mathcal{Y}}}P(x,y)\log a_xb_y  \\
 & = -\sum_{x\in\mathcal{X}}P_X(x)\log a_x-\sum_{y\in\mathcal{Y}}P_Y(y)\log b_y  \\
 & \ge \entropy_A(P_X)+\entropy_B(P_Y) \ge \entropy_C(P),
\end{split}
\end{equation}
therefore $a\otimes b\in C$. Suppose now that $\entropy_A(P_X)+\entropy_B(P_Y)\le\entropy_C(P)+\mutualinformation(X:Y)_P$ for all $P$. Then
\begin{equation}
\begin{split}
\entropy_{C^*}(P)
 & = \entropy(P)-\entropy_C(P)  \\
 & \le \entropy(P)-\entropy_A(P_X)-\entropy_B(P_Y)+\mutualinformation(X:Y)_P  \\
 & \le \entropy(P_X)-\entropy_A(P_X)+\entropy(P_Y)-\entropy_B(P_Y)  \\
 & = \entropy_{A^*}(P_X)+\entropy_{B^*}(P_Y).
\end{split}
\end{equation}
The same argument as above shows $A^*\otimes B^*\subseteq C^*$, therefore $C\subseteq(A^*\otimes B^*)^*$.
\end{proof}

We introduce two additional operations on convex corners.
\begin{definition}
Let $A$ be a convex corner on $\mathcal{X}$ and $B$ a convex corner on $\mathcal{Y}$. Their direct sum $A\oplus B$ is the smallest convex corner on $\mathcal{X}\sqcup\mathcal{Y}$ containing $a\oplus 0$ and $0\oplus b$ for all $a\in A$ and $b\in B$.

If $f:\mathcal{X}\to\mathcal{Y}$ is a function and $B$ is a convex corner on $\mathcal{Y}$ then we define the pullback
\begin{equation}
f^*(B):=\setbuild{a\in\mathbb{R}_{\ge 0}^{\mathcal{X}}}{\exists b\in B:\forall x\in\mathcal{X}:a_x\le b_{f(x)}}.
\end{equation}
\end{definition}
Note that $A\oplus B$ is the same as the convex hull of $(A\oplus 0)\cup(0\oplus B)$.

\begin{proposition}\label{prop:cornersum}
Let $A$ and $B$ be convex corners on $\mathcal{X}$ and $\mathcal{Y}$, $P_X\in\distributions(\mathcal{X})$, $P_Y\in\distributions(\mathcal{Y})$ and $p\in[0,1]$. Then
\begin{equation}
\entropy_C(pP_X\oplus(1-p)P_Y)=p\entropy_A(P_X)+(1-p)\entropy_B(P_Y)+h(p).
\end{equation}
\end{proposition}
\begin{proof}
Let $a\in A\cap\mathbb{R}_{>0}^{\mathcal{X}}$, $b\in B\cap\mathbb{R}_{>0}^{\mathcal{Y}}$ and $\lambda\in(0,1)$. Then
\begin{equation}
\begin{split}
& \langle pP_X\oplus(1-p)P_Y,\Lambda(\lambda a\oplus(1-\lambda)b)\rangle  \\
 & = -\sum_{x\in\mathcal{X}}pP_X(x)\log(\lambda a_x)-\sum_{y\in\mathcal{Y}}(1-p)P_Y(y)\log((1-\lambda) b_y)  \\
 & = -p\log\lambda+p\langle P_X,\Lambda(a)\rangle-(1-p)\log(1-\lambda)+(1-p)\langle P_Y,\Lambda(b)\rangle  \\
 & = p\langle P_X,\Lambda(a)\rangle+(1-p)\langle P_Y,\Lambda(b)\rangle+h(p)+d(p\|\lambda).
\end{split}
\end{equation}
The infimum of the last term is $0$ (attained as $\lambda\to p$). The minimum values of the first two terms in $a$ and $b$ are $p\entropy_A(P_X)$ and $(1-p)\entropy_B(P_Y)$.
\end{proof}

\begin{proposition}\label{prop:cornerpullback}
Let $\mathcal{X},\mathcal{Y}$ be finite sets, $B$ a convex corner on $\mathcal{Y}$, $f:\mathcal{X}\to\mathcal{Y}$ a function. Then for any $P\in\distributions(\mathcal{X})$ we have
\begin{equation}
\entropy_{f^*(B)}(P)=\entropy_B(f_*(P)).
\end{equation}
\end{proposition}
\begin{proof}
\begin{equation}
\begin{split}
\entropy_B(f_*(P))
 & = -\min_{b\in B}\sum_{y\in\mathcal{Y}}f_*(P)(y)\log b_y  \\
 & = -\min_{b\in B}\sum_{y\in\mathcal{Y}}\sum_{x\in f^{-1}(y)}P(x)\log b_y  \\
 & = -\min_{b\in B}\sum_{x\in\mathcal{X}}P(x)\log b_{f(x)}  \\
 & = -\min_{a\in f^*(B)}\sum_{x\in\mathcal{X}}P(x)\log a_x  \\
 & = \entropy_{f^*(B)}(P)
\end{split}
\end{equation}
\end{proof}

\begin{theorem}
A function $F:\probgraphs\to\mathbb{R}$ is the probabilistic refinement of a logarithmic spectral point $f$ if and only if $F(G,P)=\entropy_{C_f(G)}(P)$ for some function $C_f$ mapping each graph $G$ to a convex corner on $V(G)$ which satisfies the following properties for all $G,H\in\graphs$:
\begin{enumerate}[({C}1)]
\item\label{it:cornernormalised} $C_f(K_1)=[0,1]$
\item\label{it:cornermultiplicative} $C_f(G)\otimes C_f(H)\subseteq C_f(G\strongproduct H)\subseteq(C_f(G)^*\otimes C_f(H)^*)^*$
\item\label{it:corneradditive} $C_f(G\sqcup H)=C_f(G)\oplus C_f(H)$
\item\label{it:cornermonotone} if $\varphi:\complement{H}\to\complement{G}$ is a homomorphism then $\varphi^*(C_f(G))\subseteq C_f(H)$.
\end{enumerate}
\end{theorem}
\begin{proof}
As noted above, $\entropy_{C_f(G)}$ is always concave, which is \cref{it:probconcave}. Conversely, if $F$ is a probabilistic refinement then by \Cref{prop:probabilisticrefinement} and \Cref{prop:entropycharacterization} it is the entropy function of some convex corner. The equivalence of \cref{it:cornermultiplicative,it:probmultiplicative} is \Cref{prop:cornerproduct}. The equivalence of \cref{it:corneradditive,it:probadditive} follows from \Cref{prop:cornersum}. The equivalence of \cref{it:cornermonotone,it:probmonotone} is a consequence of \Cref{prop:cornerpullback} and \Cref{lem:subcorner}.
\end{proof}

We finish with a characterization of perfect graphs.
\begin{proposition}
Let $G$ be a graph. The following are equivalent
\begin{enumerate}[(i)]
\item\label{it:perfect} $G$ is perfect
\item\label{it:evGPconstant} The evaluation function $\ev_{G,P}:\Delta(\graphs)\to\mathbb{R}$, defined as $\ev_{G,P}(f)=f(G,P)$ is constant for all $P\in\distributions(V(G))$
\item\label{it:convexcornerconstant} $C_f(G)$ is the same convex corner for all $f\in\Delta(\graphs)$.
\end{enumerate}
\end{proposition}
\begin{proof}
\ref{it:perfect}$\implies$\ref{it:evGPconstant}: $G^{\strongproduct n}$ is perfect for every $n$, therefore $\alpha(\typegraph{G}{km}{P})=f(\typegraph{G}{km}{P})=\complement{\chi}(\typegraph{G}{km}{P})$ for every $f\in\Delta(\graphs)$ and $P\in\distributions[m](V(G))$. Thus for any $P\in\distributions[\mathbb{Q}](V(G))$ the limit
\begin{equation}
\lim_{k\to\infty}\frac{1}{km}\log f(\typegraph{G}{km}{P})
\end{equation}
is independent of $f\in\Delta(\graphs)$.

\ref{it:evGPconstant}$\implies$\ref{it:convexcornerconstant}: This follows from the fact that a convex corner is uniquely determined by its entropy function.

\ref{it:convexcornerconstant}$\implies$\ref{it:perfect}: $\thetabody(\complement{G})=C_\vartheta(G)=C_{\complement{\chi}_f}(G)=\fractionalvertexpackingpolytope(G)$, therefore $G$ is perfect \cite[(3.12) Corollary]{grotschel1986relaxations}.
\end{proof}

\section*{Acknowledgements}
I thank Jeroen Zuiddam and Mil\'an Mosonyi for useful discussions.
This research was supported by the National Research, Development and Innovation Fund of Hungary within the Quantum Technology National Excellence Program (Project Nr.~2017-1.2.1-NKP-2017-00001) and via the research grants K124152, KH~129601. 

\appendix

\section{Proofs}\label{sec:proof}

\begin{proof}[Proof of \cref{eq:probrefinementsequence}]
$f(G,P_n)\ge \sqrt[n]{f(\typegraph{G}{n}{P_n})}$ by \Cref{prop:probabilisticrefinement} and $f(G,P)$ is continuous in $P$, therefore
\begin{equation}
f(G,P)=\limsup_{n\to\infty}f(G,P_n)\ge\limsup_{n\to\infty}\sqrt[n]{f(\typegraph{G}{n}{P_n})}.
\end{equation}

For the other direction, let $\epsilon>0$ be small, choose some $m\in\mathbb{N}$ and distribution $Q\in\distributions[m](V(G))$ such that $\norm[1]{P-Q}<\epsilon/2$ and $\support Q=\support P$. When $n$ is large enough, $\norm[1]{P_n-P}<\epsilon/2$ and therefore
\begin{equation}
\begin{split}
k_n
 & :=\left\lfloor\min_{v\in \support Q}\frac{nP_n(v)}{mQ(v)}\right\rfloor\ge\min_{v\in V(G)}\frac{nP_n(v)}{mQ(v)}-1  \\
 & \ge\min_{v\in \support Q}\frac{n(Q(v)-\epsilon)}{mQ(v)}-1=\frac{n}{m}\left(1-\frac{\epsilon}{\min_{v\in \support Q}Q(v)}\right)-1
\end{split}
\end{equation}
goes to $\infty$ as $n\to\infty$. This choice ensures $nP_n-k_nmQ\ge 0$, therefore we can apply \Cref{prop:convexcombinationbounds} to get
\begin{equation}
\typegraph{G}{k_nm}{Q}\le\typegraph{G}{k_nm}{Q}\strongproduct\typegraph{G}{(n-k_nm)}{\frac{nP_n-k_nmQ}{n-k_nm}}\le\typegraph{G}{n}{P_n},
\end{equation}
which in turn implies
\begin{equation}
\left(\sqrt[k_nm]{f(\typegraph{G}{k_nm}{Q})}\right)^{\frac{k_nm}{n}}\le\sqrt[n]{f(\typegraph{G}{n}{P_n})}.
\end{equation}
As $n\to\infty$, the inequality becomes
\begin{equation}
\liminf_{n\to\infty}\sqrt[n]{f(\typegraph{G}{n}{P_n})}\ge f(G,Q)^{\lim_{n\to\infty}\frac{k_nm}{n}}=f(G,Q)^{1-\frac{\epsilon}{\min_{v\in \support Q}Q(v)}}.
\end{equation}
Finally, let $Q\to P$ and then $\epsilon\to 0$.
\end{proof}

\begin{proposition}\label{prop:supFopen}
Let $G$ be a graph, $\mathcal{U}\subseteq\distributions(V(G))$ be open and consider the subsets $\typeclass{n}{\mathcal{U}}=\bigcup_{P\in \mathcal{U}\cap\distributions[n](V(G))}\typeclass{n}{P}$. Then
\begin{equation}
\sup_{P\in \mathcal{U}}f(G,P)=\lim_{n\to\infty}\sqrt[n]{f(\typegraph{G}{n}{\mathcal{U}})}.
\end{equation}
\end{proposition}
\begin{proof}
Let $Q$ be in the closure of $\mathcal{U}$ such that $f(G,Q)=\sup_{P\in \mathcal{U}}f(G,P)$ and choose a sequence $P_n\in \mathcal{U}\cap\distributions[n](V(G))$ such that $P_n\to Q$. By \cref{eq:probrefinementsequence} we get
\begin{equation}
\sup_{P\in \mathcal{U}}f(G,P)=\lim_{n\to\infty}\sqrt[n]{f(\typegraph{G}{n}{P_n})}\le\liminf_{n\to\infty}\sqrt[n]{f(\typegraph{G}{n}{\mathcal{U}})}.
\end{equation}

For the other inequality, applying $f$ to
\begin{equation}
\typegraph{G}{n}{\mathcal{U}}\le\bigsqcup_{P\in \mathcal{U}\cap\distributions[n](V(G))}\typegraph{G}{n}{P}
\end{equation}
gives the upper bound
\begin{equation}
\begin{split}
\limsup_{n\to\infty}\sqrt[n]{f(\typegraph{G}{n}{\mathcal{U}})}
 & \le \limsup_{n\to\infty}\sqrt[n]{\sum_{P\in \mathcal{U}\cap\distributions[n](V(G))}f(\typegraph{G}{n}{P})}  \\
 & \le \limsup_{n\to\infty}\sqrt[n]{|\distributions[n](V(G))|\max_{P\in \mathcal{U}\cap\distributions[n](V(G))}f(\typegraph{G}{n}{P})}  \\
 & = \lim_{n\to\infty}\sqrt[n]{|\distributions[n](V(G))|}\limsup_{n\to\infty}\sqrt[n]{\max_{P\in \mathcal{U}\cap\distributions[n](V(G))}f(\typegraph{G}{n}{P})}  \\
 & \le \limsup_{n\to\infty}\max_{P\in \mathcal{U}\cap\distributions[n](V(G))}f(G,P) \le \sup_{P\in \mathcal{U}}f(G,P).
\end{split}
\end{equation}
\end{proof}

\begin{proof}[Proof of \cref{eq:probrefinementtypical}]
By \Cref{prop:supFopen} and continuity of $f(G,\cdot)$,
\begin{equation}
\lim_{\epsilon\to 0}\lim_{n\to\infty}\frac{1}{n}\log f(\typegraph{G}{n}{\ball{\epsilon}{P}})
 = \log\lim_{\epsilon\to 0}\sup_{Q\in\ball{\epsilon}{P}}f(G,Q)
 = F(G,P).
\end{equation}
\end{proof}

\begin{proof}[Proof of \cref{eq:probrefinementessential}]
For any $\epsilon>0$ we have $P^{\otimes n}(\typeclass{n}{\ball{\epsilon}{P}})\to 1$ by the weak law of large numbers, thus
\begin{equation}
\limsup_{n\to\infty}\min_{\substack{S\subseteq V(G)^n  \\  P^{\otimes n}(S)>c}}\frac{1}{n}\log f(G^{\strongproduct n}[S])
 \le \lim_{n\to\infty}\frac{1}{n}\log f(\typegraph{G}{n}{\ball{\epsilon}{P}}).
\end{equation}
The limit of the right hand side as $\epsilon\to 0$ is still an upper bound and equals $F(G,P)$ by \cref{eq:probrefinementtypical}.

Let $S_n\subseteq V(G)^n$ with $P^{\otimes n}(S_n)>c>0$ minimize $f(G^{\strongproduct n}[S_n])$ and let $(\epsilon_n)_n$ be such that $\epsilon\to 0$ and $P^{\otimes n}(\typeclass{n}{\ball{\epsilon_n}{P}})\to 1$. Then
\begin{equation}
\liminf_{n\to\infty}P^{\otimes n}(S\cap\typeclass{n}{\ball{\epsilon_n}{P}})\ge \liminf_{n\to\infty}P^{\otimes n}(S_n)+P^{\otimes n}(\typeclass{n}{\ball{\epsilon_n}{P}})-1\ge c,
\end{equation}
therefore for any $0<c'<c$ and large enough $n$ we have $P^{\otimes n}(S_n\cap\typeclass{n}{\ball{\epsilon_n}{P}})>c'$. Let $P_n\in\distributions[n](V(G))\cap\ball{\epsilon_n}{P}$ be one such that $P^{\otimes n}(S_n\cap\typeclass{n}{P_n})$ is maximal. Then
\begin{equation}
\frac{|S_n\cap\typeclass{n}{P_n}|}{|\typeclass{n}{P_n}|}=\frac{P^{\otimes n}(|S_n\cap\typeclass{n}{P_n}|)}{P^{\otimes n}(|\typeclass{n}{P_n}|)}\ge\frac{c'}{(n+1)^{|V(G)|}}
\end{equation}
We can use \Cref{lem:transitiveinduced} for the vertex-transitive graph $\typegraph{G}{n}{P_n}$ and the subset $S_n\cap\typeclass{n}{P_n}$, which says that
 $\typegraph{G}{n}{P_n}\le\complement{K_N}\strongproduct G^{\strongproduct n}[S_n\cap\typeclass{n}{P_n}]$ with
\begin{equation}
N
 =\left\lfloor\frac{|\typeclass{n}{P_n}|}{|S_n\cap\typeclass{n}{P_n}|}\ln|\typeclass{n}{P_n}|\right\rfloor+1
 \le\frac{(n+1)^{|V(G)|}}{c'\log e}n\entropy(P_n)+1.
\end{equation}
Using that $f$ is monotone, multiplicative and normalized, this implies
\begin{equation}
\begin{split}
\liminf_{n\to\infty}\min_{\substack{S\subseteq V(G)^n  \\  P^{\otimes n}(S)>c}}\frac{1}{n}\log f(G^{\strongproduct n}[S])
 & = \liminf_{n\to\infty}\frac{1}{n}\log f(G^{\strongproduct n}[S_n])  \\
 & \ge \liminf_{n\to\infty}\frac{1}{n}\log f(G^{\strongproduct n}[S_n\cap\typeclass{n}{P_n}])  \\
 & \ge \liminf_{n\to\infty}\frac{1}{n}\log \frac{f(\typegraph{G}{n}{P_n})}{\frac{(n+1)^{|V(G)|}}{c'\log e}n\entropy(P_n)+1}=F(G,P).
\end{split}
\end{equation}
\end{proof}

\begin{proof}[Proof of \Cref{prop:ffromF}]
Choose $\mathcal{U}=\distributions(V(G))$ in \Cref{prop:supFopen}. Then $\typeclass{n}{\mathcal{U}}=V(G)^n$ and by multiplicativity of $f$ we have
\begin{equation}
\max_{P\in\distributions(V(G))}f(G,P)=\lim_{n\to\infty}\sqrt[n]{f(G^{\strongproduct n})}=f(G).
\end{equation}
\end{proof}

\begin{proof}[Proof of \Cref{prop:Fstarproperties}]
\begin{equation}
\begin{split}
F^*(G\sqcup H,pP_G\oplus(1-p)P_H)
 & = \entropy(pP_G\oplus(1-p)P_H)-F(\complement{G\sqcup H},pP_G\oplus(1-p)P_H)  \\
 & = p\entropy(P_G)+(1-p)\entropy(P_H)+h(p)-F(\complement{G}+\complement{H},pP_G\oplus(1-p)P_H)  \\
 & = p\entropy(P_G)+(1-p)\entropy(P_H)+h(p)-pF(\complement{G},P_G)-(1-p)F(\complement{H},P_H)  \\
 & = pF^*(G,P_G)+(1-p)F^*(H,P_H)+h(p).
\end{split}
\end{equation}
\begin{equation}
\begin{split}
F^*(G+H,pP_G\oplus(1-p)P_H)
 & = \entropy(pP_G\oplus(1-p)P_H)-F(\complement{G+H},pP_G\oplus(1-p)P_H)  \\
 & = p\entropy(P_G)+(1-p)\entropy(P_H)+h(p)-F(\complement{G}\sqcup\complement{H},pP_G\oplus(1-p)P_H)  \\
 & = p\entropy(P_G)+(1-p)\entropy(P_H)+h(p)-pF(\complement{G},P_G)-(1-p)F(\complement{H},P_H)-h(p)  \\
 & = pF^*(G,P_G)+(1-p)F^*(H,P_H).
\end{split}
\end{equation}

\begin{equation}
\begin{split}
F^*(G\costrongproduct H,P)
 & = \entropy(P)-F(\complement{G\costrongproduct H},P)  \\
 & = \entropy(P)-F(\complement{G}\strongproduct\complement{H},P)  \\
 & \ge \entropy(P)-F(\complement{G},P_G)-F(\complement{H},P_H)  \\
 & = \entropy(P)-\entropy(P_G)+F^*(G,P_G)-\entropy(P_H)+F^*(H,P_H)  \\
 & = F^*(G,P_G)+F^*(H,P_H)-\mutualinformation(G:H)_P.
\end{split}
\end{equation}

\begin{equation}
\begin{split}
F^*(G\costrongproduct H,P)
 & = \entropy(P)-F(\complement{G\costrongproduct H},P)  \\
 & = \entropy(P)-F(\complement{G}\strongproduct\complement{H},P)  \\
 & \le \entropy(P)-F(\complement{G},P_G)-F(\complement{H},P_H)+\mutualinformation(G:H)_P  \\
 & = \entropy(P)-\entropy(P_G)+F^*(G,P_G)-\entropy(P_H)+F^*(H,P_H)+\mutualinformation(G:H)_P  \\
 & = F^*(G,P_G)+F^*(H,P_H).
\end{split}
\end{equation}
\end{proof}

\bibliography{refs}{}

\end{document}